\documentclass[12pt,twoside,reqno]{amsart}
\usepackage[colorlinks=true,citecolor=blue]{hyperref}
\usepackage{mathptmx, amsmath, amssymb, amsfonts, amsthm, mathptmx, enumerate, color}
\usepackage{graphicx}
\usepackage{epstopdf}

\newtheorem{theorem}{Theorem}[section]
\newtheorem{corollary}{Corollary}[section]

\newtheorem{proposition}{Proposition}[section]
\theoremstyle{definition}
\newtheorem{definition}{Definition}[section]
\newtheorem{example}{Example}[section]

\newtheorem{remark}{Remark}[section]

\numberwithin{equation}{section}

\def\cl{\mbox{\rm cl}\,}

\def\bd{\mbox{\rm bd}\,}
\def\tbd{\mbox{\scriptsize {\rm bd}}}

\def\epi{\mbox{\rm epi}\,}
\def\tepi{\mbox{\scriptsize {\rm epi}}\,}
\def\tgph{\mbox{\scriptsize {\rm gph}}\,}
\def\gph{\mbox{\rm gph}\,}
\def\dom{\mbox{\rm dom}\,}
\def\cone{\mbox{\rm cone}\,}

\def \N{\mathbb{N}}
\def \R{\mathbb{R}}
\def \B{\mathbb{B}}

\DeclareMathOperator*{\Limsup}{Lim\, sup}

\begin{document}
\setcounter{page}{1}

\vspace*{1.0cm}
 \title[Coderivatives with respect to a set of the normal cone mappings]
 {Coderivatives with respect to a set of the normal cone mappings and their applications}
 \author[V.D. Thinh, X. Qin, J.C. Yao]{Vo Duc Thinh$^{1,2}$, Xiaolong Qin$^{1,3,*}$, Jen-Chih Yao$^4$}
\maketitle
\vspace*{-0.6cm}

\begin{center}
{\footnotesize {\it

 $^1$School of Mathematical Sciences, Zhejiang Normal University, Jinhua, 321004,  China\\
  $^{2}$Dong Thap University, Cao Lanh City, 81000, Dong Thap, Vietnam\\
 $^3$Center for Advanced Information Technology,  Kyung Hee University, Seoul, Korea\\
 $^4$Research Center for Interneural Computing,   China Medical University, Taichung 40402, Taiwan
}
}\end{center}

\vskip 4mm {\small\noindent {\bf Abstract.}
Establishing explicit formulas of coderivatives with respect to a set of the normal cone mapping to a polyhedron, the solution set of a variational inequalities system, is one of the main goals of this paper. By using our coderivative formulas, we provide a characteristic of the Aubin property with respect to a set of that normal cone mapping and thereby give necessary optimality conditions for simple bilevel optimization problems under weak qualification conditions.

\noindent {\bf Keywords.}
 Aubin property with respect to a set; Bilevel programming; Coderivative with respect to a set; Linear inequality system; Optimality condition.

\noindent {\bf 2020 Mathematics Subject Classification.} 49J53, 90C30, 90C31.}

\renewcommand{\thefootnote}{}
\footnotetext{ $^*$Corresponding author.
 \par
 E-mail address: qxlxajh@163.com (X. Qin).}

\section{ Introduction}

The main goal of this paper is to establish explicit formulas of coderivatives with respect to $\mathcal C$ of the normal cone mapping to $\Theta,$ where $\Theta$ and $\mathcal C\subset \R^n$ are convex polyhedrons, respectively, given by
\begin{equation*}
    \Theta:=\left\{x\in \R^n\mid \langle a_i, x\rangle\le b_i\; \forall i\in I_1 \right\}
\end{equation*}
and
\begin{equation*}
    \mathcal C:=\left\{ x\in \R^n\mid \langle a_i, x\rangle\le b_i\; \forall i\in I_2\right\},
\end{equation*}
with $I_1:=\{1,\ldots, \ell_1\}$, $I_2:=\left\{1,\ldots, \ell_2\right\}$, and $a_i\in \R^n$ and  $b_i\in \R$ for each $i\in I_1\cup I_2.$ By using these formulas, we provide a characteristic of the Aubin/Lipschitz-like property with respect to $\mathcal C$ of the normal cone mapping to $\Theta$, denoted by $\mathcal N:\mathbb R^n\to \mathbb R^n$ and defined by $\mathcal N(x):=N(x,\Theta)$ for any $x\in \R^n.$ According to our coderivative formulas, we also give simple and sharp optimality conditions for the bilevel optimization problem, given in the following form:
\begin{align}\label{SBLP}
    \min\; &f(x)\tag{SBLP}\\
    \text{ \rm such that } & x\in \mathcal S\cap \mathcal C\nonumber\\
    \text{ \rm where }& \mathcal S:={\rm arcmin}\, \{c^{\top}x\mid x\in \Theta\}.\nonumber
\end{align} Here $c\in \R^n,$ $f:\Theta\to \bar{\R}:=\R\cup\{\infty\}$.

The calculation of the coderivative for the normal cone maping to $\Theta$ in finite-dimensional space was first established by   Henrion et al. in \cite{HOS2009} and then was further studied in infinite-dimensional space in \cite{HMN2010}. In that paper, the authors applied the obtained coderivative formulas to derive efficient conditions for robust Lipschitzian stability of solution maps to parameterized variational inequalities via the coderivative characterization of the major Lipschitz-like/Aubin property for general set-valued mappings  with computing the exact bound of Lipschitzian moduli. Some generalized results of \cite{HMN2010} can be found in \cite{Nam10,Qui13}.

In this paper, inspired by \cite{HMN2010,HOS2009}, we first compute coderivatives with respect to $\mathcal C$ of $\mathcal N$ and then we provide an explicit characteristic of the Aubin property with respect to $\mathcal C$ of $\mathcal N$ via the new coderivative formulas. The Aubin property with respect (or relative) to a set of a multifunction was introduced in \cite{Roc98} and recently  continued to be studied in \cite{MLYY23, MWY23, TQ01-ZNU,TQ02-ZNU,YY22}. More specifically, the authors in \cite{MLYY23} established a Mordukhovich type criterion for Aubin property with respect to a set of multifunctions with the aid of projectional coderivatives. Using this approach, the authors in \cite{YY22} provided the sufficiently condition for the {\it relative Lipschitz-like property} (the Aubin property with respect to the domain of mapping) of the solution mapping of affine variational inequalities. When the domain of the solution mapping is a polyhedron, the necessity is also obtainable. In particular, the explicit criterion for the relative Lipschitz-like property of the solution mapping of a linear complementarity problem with respect to its domain is presented. However, it should be known that the relative Lispchitz-like property of the solution mapping of affine variational inequalities in that paper is very difficult to obtain at boundary points of the domain of that solution mapping if the interior of that domain is nonempty. The reason is that in this case $\mathcal N(x)=0$ for any $x$ belongs to the interior of the domain of the solution mapping, which is totally different from that one at boundary points. Moreover, in that paper, the formulas of the projectional coderivative of the solution mapping to variational inequalities systems were provided in implicit forms, and  further applications of the relative Lipschitz-like property, for example in the study of optimality conditions, were not presented in that paper. Recently, the authors in \cite{MWY23} presented coderivatives with respect to a set for multifunctions in Banach spaces and used them as tools to obtain a new Mordukhovich criterion for the Aubin property/Lipschitz-like with respect to a set of multifunctions.
 Almost at the same time as \cite{MWY23},  the authors in \cite{TQ01-ZNU} also stated a new Mordukhovich criterion for the Aubin property with respect to a set of multifunctions due to the limiting coderivative with respect to that set by another approach.  Formulas for calculating and some applications in studying optimality conditions of the limiting coderivative  presented in \cite{TQ02-ZNU}. Moreover, with that approach, the applications in studying optimality conditions of the Aubin property with respect to a set of multifunctions were introduced in \cite{TQ02-ZNU}.  It is important to note that although there is no reference to the paper \cite{MWY23}, there are some similarities the approaches in \cite{TQ01-ZNU, TQ02-ZNU} and \cite{MWY23}. Besides, in finite-dimensional spaces, the limiting coderivative with respect to a convex polyhedron of multifunctions in \cite{TQ01-ZNU} coincides that one in \cite{MWY23}. However, in general, the limiting coderivative with respect to a set in \cite{TQ01-ZNU, TQ02-ZNU} and \cite{MWY23} are different.

Another main goal of this paper is to provide optimality conditions for \ref{SBLP} problem due to the Aubin property and the limiting coderivative with respect to $\mathcal C$ of $\mathcal N$. \ref{SBLP} problem is a generalization of linear bilevel programming and is called a simple bilevel optimization problem with the linear lower-level problem. In the case that  the objective function of the upper-level problem is a convex function,  \ref{SBLP} problem is reduced to a particular case of the simple convex bilevel programming problem (SCBP) studied in \cite{DDD2010, DDDP2020, DZ2020}. Optimality conditions for the SCBP problem were studied in \cite{DDD2010} under several constraint qualifications and further discussed  in \cite{DDDP2020} by expressing the SCBP problem as a simple  mathematical programming problem under equilibrium constraints (SMPEC). However, those optimality conditions may not applied for the SBLP problem because of the loss of convexity of the objective function. By writing the SBLP problem in the following form:
\begin{align}\label{SBLP1}
    \min\; &f(x)\\
    \text{ \rm such that } & 0\in -c+ N(x,\Theta) \text{ \rm and } x\in \mathcal C,\nonumber
\end{align}
and by using the obtained coderivative formulas, we   provide new optimality conditions for the SBLP problem under weak qualification conditions. The highlights of the paper are as follows:
\begin{itemize}
\item Compute the normal cones with respect to $\mathcal C\times\R^n$ of  set $\gph \mathcal N$.
\item Establish formulas of limiting coderivative with respect to $\mathcal C$ of $\mathcal N$ from the formulas of the limiting normal cone with respect to $\mathcal C\times\R^n$ of $\gph \mathcal N.$
\item Provide explicit characteristics of the Aubin property with respect to $\mathcal C$ of $\mathcal N.$
\item Apply the formulas of limiting coderivative with respect to $\mathcal C$ of $\mathcal N$ to the  optimality conditions for  problems \eqref{SBLP1} and  \ref{SBLP}.
\end{itemize}

\section{Preliminaries}

From now on, we always assume that all of the  spaces considered in this paper are finite dimension spaces with   norm $\Vert\cdot\Vert$ and scalar product $\langle \cdot,\cdot\rangle.$ Let $\mathbb R^n, \mathbb R^m, \R^s$ be finite dimmension spaces with norms $\Vert \cdot\Vert_{n}$, $\Vert\cdot\Vert_m$, and $\Vert\cdot\Vert_s$, respectively. The Cartesian product $\mathbb R^n\times \mathbb R^m$ is equipped a norm defined by $\Vert (x,y)\Vert:=\Vert \cdot\Vert_n +\Vert\cdot\Vert_m.$ We sometimes use $\Vert\cdot\Vert$ for some norm if there is no the confusion. The closed unit ball in $\R^s$ is signified by $\B$ while $\B(x,r)$ is borrowed  to denote the closed ball centered at $x$ with radius $r>0.$
Define
$$
\R^s_+:=\left\{x=(x_1,\ldots, x_s)\in\R^s\mid x_i\ge 0\; \forall i=1,\ldots, s\right\}.
$$
Given a real numbers sequence $(t_k)$, we write $t_k\to 0^+$ if $t_k\to 0$ and $t_k\ge 0$ for all $k,$ while $t_k\downarrow 0$ means that $t_k\to 0^+$ and $t_k>0$ for all $k\in \N.$

Let $\Omega$ be a nonempty set in $\mathbb R^s$ and $\bar x\in\Omega$. Denote the {\it interior}, the {\it closure}, and the {\it boundary} of $\Omega$  by ${\rm int}\, \Omega,$ $\cl \Omega$, and $\bd \Omega$, respectively.  Denote $\mathcal R(\bar x,\Omega):=\left\{x^*\mid \exists p>0: \bar x+px^*\in \Omega\right\}=\R_+(\Omega-\bar x).$
We use the notion $x_k\xrightarrow{\Omega}\bar x$ to say that $x_k\to \bar x$ and $x_k\in \Omega$ while $x_k\xrightarrow[x_k\ne \bar x]{\Omega}\bar x$ means that $x_k\xrightarrow{\Omega}\bar x$ and $x_k\ne \bar x$ for all $k\in \N$. Let $\R^s$ be equipped the Euclidean norm. The {\it distance function} to $\Omega$, $d_{\Omega}:\mathbb R^s\to \mathbb R,$ is defined by $d_{\Omega}(x):=\inf_{u\in \Omega}\Vert u-x\Vert\; \text{ \rm for all } x\in \mathbb R^s.$ An element $u$ in $\Omega$ satisfying $d_{\Omega}(x)=\Vert u-x\Vert$ is called a (Euclidean) {\it projector} (or {\it closest point}) of $x$ onto $\Omega.$ The multifunction $\Pi(\cdot, \Omega): \mathbb R^s\rightrightarrows\mathbb R^s, x\mapsto \Pi(x,\Omega):=\left\{u\in \mathbb R^s\mid \Vert u-x \Vert=d_{\Omega}(x)\right\}$ is called the {\it set-valued mapping projection} onto $\Omega$, and the set $\Pi(x,\Omega)$ is called the {\it Euclidean projector set} of $x$ onto $\Omega.$ Note that the set $\Pi(x,\Omega)$ can be empty, however, if $\Omega$ is closed, then $\Pi(x,\Omega)\ne \emptyset$ for any $x\in \mathbb R^s.$ Given $u\in \Omega,$ we define $\Pi^{-1}(u,\Omega):=\left\{x\in \mathbb R^s\mid u\in \Pi(x,\Omega) \right\}.$
The {\it proximal and Fréchet normal cones} to $\Omega$ at $\bar x\in\Omega$ are respectively  given (see \cite[Definition~1.1 and page 240]{Mor06}) by
$$
N^p(\bar x,\Omega):=\cone[\Pi^{-1}(\bar x,\Omega)-\bar x]
$$
and
$$
\hat N(\bar x,\Omega):=\left\{x^*\in \mathbb R^n\mid \limsup_{x\xrightarrow{\Omega}\bar x}\dfrac{\langle x^*, x-\bar x\rangle}{\Vert x-\bar x\Vert}\le 0\right\}.
$$
The {\it limiting normal cone} to $\Omega$ at $\bar x$ is defined by
$$
N(\bar x,\Omega):=\Limsup_{x\to\bar x}\Big(\cone[x-\Pi(x,\Omega)]\Big).
$$
The {\it tangent/contingent cone} to $\Omega$ at $\bar x$ is given by
$$
T(\bar x,\Omega):=\left\{v\in\R^n\mid \exists t_k\downarrow 0, v_k\rightarrow\bar v \text{ \rm such that } \bar x+t_kv_k\in \Omega, \forall k\in\N \right\}.
$$

Let $f:\R^n\to\bar\R:=\R\cup\{\infty\}$ be an extended real-valued mapping. We respectively denote the domain and the epi-graph of $f$ by
$$
\dom f:=\left\{x\in\R^n\mid f(x)<\infty\right\} \text{ \rm and } \epi f:=\left\{(x,\alpha)\mid \alpha\ge f(x)\right\}.
$$
Given $\mathcal C\subset \R^n,$ we define $f_{\mathcal C}:\R^n\to\bar\R$ by $f_{\mathcal C}(x):=\begin{cases}f(x)&\text{ \rm if } x\in \mathcal C,\\
\infty &\text{ \rm otherwise.}\end{cases}$\\
The mapping $f$ is called {\it lower semi-continuous} with respect to (on) $\mathcal C\subset \dom f$ if $\epi f_{\mathcal C}$ is a closed set. The set of all of the mappings which are lower semi-continuous with respect to $\mathcal C$ is denoted by $\mathcal F(\mathcal C).$ If $f\in\mathcal F(\dom f)$, then we say that $f$ is lower semi-continuous.

Consider the set-valued mapping $F:\mathbb R^n\rightrightarrows\mathbb R^m.$ We use
$$\dom F:=\left\{x\in \R^n\mid F(x)\ne \emptyset\right\}\; \text{ \rm and } \gph F:=\left\{(x,y)\in\R^n\times\R^m\mid y\in F(x)\right\}$$ to stand for  the {\it domain} and {\it graph} of $F.$ The {\it sequential Painlev\'{e}-Kuratowski upper/outer limit} of $F$ at $\bar x\in \dom F$ is given by
$$
\Limsup_{x\to\bar x}F(x):=\left\{y\in \R^m\mid \exists x_k\to \bar x, y_k\to y \text{ \rm with } y_k\in F(x_k) \; \text{ \rm for all } k\in \N\right\}.
$$
Given $\mathcal C\subset \R^n,$ we define the set-valued mapping $F_{\mathcal C}:\R^n\rightrightarrows\R^m$ as follows:
$$F_{\mathcal C}(x):=\begin{cases}
F(x)& \text{ \rm if } x\in \mathcal C,\\
\emptyset & \text{ \rm otherwise.}
\end{cases}$$
Thus
$$
\dom F_{\mathcal C}=\dom F\cap\mathcal C \; \text{ \rm and } \;  \gph F_{\mathcal C} = \gph F\cap(\mathcal C\times\R^m).
$$
The {\it epigraphical/profile mapping} of $F$, denoted by $\mathcal E^F:\R^n\rightrightarrows\R^m,$ is defined by
$$
\mathcal E^F(x):=F(x)+\R^m_+\; \text{ \rm for all }x\in \R^n.
$$
If $F:\R^n\rightrightarrows\R$ is a singleton mapping, i.e., $F(x)=\{f(x)\}$ for all $x\in \R^n$ then $\gph\mathcal E^F=\epi f:=\left\{(x,\alpha)\mid \alpha\ge f(x)\right\}.$

The Aubin property with respect to a set (also known as locally Lipschitz-like property with respect to a set in \cite{MLYY23,YY22}) of $F$ was introduced in \cite[Definition~9.36]{Roc98} as follows.

\begin{definition}[\cite{Roc98}, Definition~9.36]
{\rm Let $\mathcal C$ be a nonempty and  closed subset of $\mathbb R^n$, and let $F:\mathbb R^n\rightrightarrows \R^m$. Then $F$ is said  to have the {\it Aubin property with respect to} $\mathcal C$ around $(\bar x,\bar y)\in \gph F$ if $\gph F$ is locally closed around $(\bar x,\bar y)$ and there exist $\kappa\ge 0$ and neighborhoods $V$ of $\bar y$ and $U$ of $\bar x$ such that
\begin{equation}\label{Aubin-pro}
    F(u)\cap V\subset F(x) +\kappa\Vert u-x\Vert \mathbb B\quad  \forall x,u\in \mathcal C\cap U.
\end{equation}
In the case that \eqref{Aubin-pro} holds for $\mathcal C=\dom F$, we say that $F$ has the {\it relative Aubin property} around $(\bar x,\bar y).$
}
\end{definition}

We now introduce the concept of local Lipschitz continuity with respect to set for single-valued functions as follows.

\begin{definition}[\cite{Roc98}, Definition~9.1~(b)] \label{lips-vec}
{\rm Let $\mathcal C$ be a nonempty and closed subset of $\mathbb R^n.$ Let $f:\mathbb R^n\to\bar\R$ and $\bar x\in \dom f.$ Then $f$ is said  to be {\it locally Lipschitz continuous with respect to} $\mathcal C$ around $\bar x$ if
\begin{equation}\label{lips-con}
{\rm lip}_{\mathcal C}f(\bar x):=\limsup_{x,u\xrightarrow[x\ne u]{\mathcal C}\bar x}\dfrac{\vert f(x)-f(u)\vert}{\Vert x-u\Vert}< \infty,
\end{equation}
where ${\rm lip}_{\mathcal C}f(\bar x)$ is called the {\it exact Lipschitzian constant with respect to}  $\mathcal C$ of $f$ at $\bar x.$ If \eqref{lips-con} holds with $\mathcal C=\dom f$, then $f$ is said  to be {\it relatively locally Lipschitz continuous} around $\bar x.$
}
\end{definition}

\begin{remark} {\rm It is easy to see that the scalar indicator mapping of a nonempty set $\Omega$ is relatively local Lipschitz countinuous around any $x\in \Omega$, but it is not local Lipschitz countinuous around boundary points of $\Omega.$ Similarly, the set-valued indicator mapping of $\Omega$ satisfies the relative Aubin property around any $(x,0)\in \gph\Delta_{\Omega}$, but it does not satisfies Aubin property around boundary points $(\bar x,0)$ for any $\bar x\in\bd\Omega.$}
\end{remark}

On the basic of the use of the limiting normal cone to sets, the limiting corderivative and subdifferential of multifunctions were introduced in \cite{Mor06,BM07}.
\begin{definition} {\rm Let $F:\R^n \rightrightarrows\R^m$ and $(\bar x,\bar y)\in \gph F.$

{\rm (i)} \cite[Definition~3.32]{Mor06} The {\it limiting coderivative} (also known as the {\it normal coderivative}) of $F$ at $(\bar x,\bar y)$ is the multifunction $D^*F(\bar x,\bar y): \R^m\rightrightarrows\R^n$, which is defined by
$$
D^*F(\bar x,\bar y)(y^*):=\left\{x^*\in\R^n\mid (x^*,-y^*)\in N((\bar x,\bar y),\gph F)\right\} \; \text{ \rm for all } y^*\in \R^m.
$$

{\rm (ii)} \cite[Definition~2.1]{BM07} The {\it limiting subdifferential}  of $F$ at $(\bar x,\bar y)$ is given by
$$
\partial F(\bar x,\bar y):=\left\{x^*\in\R^n\mid (x^*,-y^*)\in N((\bar x,\bar y),\gph \mathcal E^F), y^*\in \R^m_+, \Vert y^*\Vert=1\right\}.
$$

If $F$ is the singleton mapping $f$, then we use $D^*f(\bar x)$ and $\partial f(\bar x)$ instead of $D^*f(\bar x,f(\bar x))$ and $\partial f(\bar x,f(\bar x))$, respectively.
}
\end{definition}

 \begin{definition}\label{normal-cone-wrt}
 {\rm Let $\Omega, \mathcal C\subset \mathbb R^s$ be nonempty sets, and $\bar x\in \Omega\cap \mathcal C.$ Suppose that $\mathcal C$ is convex and locally closed around $\bar x.$ Then

{\rm (i)} The {\it Fr\'{e}chet normal cone} to $\Omega$ at $\bar x$ {\it with respect to} $\mathcal C$ is defined by
\begin{align}\label{def-fre-normal-cone}
\hat N_{\mathcal C}(\bar x,\Omega):=\left\{x^*\in \R^s\mid \exists p>0: \bar x+px^*\in \mathcal C, \limsup\limits_{x\xrightarrow{\Omega\cap\mathcal C}\bar x}\dfrac{\langle x^*, x-\bar x\rangle}{\Vert x-\bar x\Vert}\le 0\right\}.
\end{align}

{\rm (ii)} The {\it limiting normal cone} to $\Omega$ at $\bar x$ {\it with respect to} $\mathcal C$ is defined by
 \begin{equation}\label{limiting-normal-cone}
 N_{\mathcal C}(\bar x,\Omega):=\Limsup_{x\xrightarrow{\Omega\cap\mathcal C}\bar x}\hat N_{\mathcal C}(x,\Omega).\end{equation}
 }
\end{definition}

\begin{remark}
    {\rm (i) It is easy to see from \eqref{def-fre-normal-cone} that $\hat N_{\mathcal C}(\bar x,\Omega)$ does not depend on equivalent norms, which is due to the property of $\hat N(\bar x,\Omega\cap\mathcal C).$}

    {\rm (ii) In Definition~\ref{normal-cone-wrt}, if $\mathcal C$ is a neighborhood of $\bar x$, then $\hat N_{\mathcal C}(\bar x,\Omega)$ and  $N_{\mathcal C}(\bar x,\Omega)$ reduces to $\hat N(\bar x,\Omega)$ and $N(\bar x,\Omega)$, respectively.}
\end{remark}

    We next recall several foundational properties of the Fr\'{e}chet and limiting normal cones with respect to a set presented in \cite{TQ02-ZNU} as follows.
    \begin{proposition}[\cite{TQ02-ZNU}]\label{pro1}
    Let $\Omega$ be a nonempty set, and let $\mathcal C$ be a nonempty, convex, and closed set in $\mathbb R^s.$ Then, for each $\bar x\in \Omega\cap \mathcal C,$ the following assertions hold:

    {\rm (i)} $\hat N_{\mathcal C}(\bar x,\Omega)=\hat N(\bar x,\Omega\cap\mathcal C)\cap \mathcal R(\bar x,\mathcal C).$

    In addition, if $\Omega$ is convex, then
    \begin{align}
\hat N_{\mathcal C}(\bar x,\Omega)=\Big\{x^*\in \R^s\mid  \exists p>0: \bar x+px^*\in \mathcal C, \langle x^*, x-\bar x\rangle\le 0 \; \forall x\in \Omega\cap\mathcal C\Big\}.\label{formula-prox-cone-2}
\end{align}

    {\rm (ii)} $\hat N_{\mathcal C}(\bar x,\Omega)$ is a convex cone.

    {\rm (iii)} $N_{\mathcal C}(\bar x,\Omega)$ is a closed cone.

    {\rm (iv)} If $\R^s$ is an Euclidean space and $\Omega$ is locally closed around $\bar x$, then
    $$
    N_{\mathcal C}(\bar x, \Omega)= \Limsup\limits_{x\xrightarrow{\mathcal C}\bar x}\left(\cone[x-\Pi(x,\Omega\cap\mathcal C)]\right).
    $$
    \end{proposition}

We now recall the concept and some basic properties of the limiting coderivative and subdifferential with respect to a set.

\begin{definition}[\cite{TQ02-ZNU}]\label{coderivative-wrt}
 {\rm   Give a nonempty closed set $\mathcal C\subset \mathbb R^n$ and $\bar x\in \mathcal C$. Let $F:\mathbb R^n\rightrightarrows \R^m$ have a locally closed graph and $(\bar x,\bar y)\in \gph F.$ The {\it limiting coderivative with respect to} $\mathcal C$ of $F$ at $(\bar x,\bar y)$ is a multifunction $D_{\mathcal C}^*F(\bar x,\bar y):\mathbb R^m\rightrightarrows \mathbb R^n$ defined by
    $$D^*_{\mathcal C}F(\bar x,\bar y)(y^*) = \left\{ x^*\in\mathbb R^n\mid (x^*,-y^*)\in N_{\mathcal C\times \mathbb R^m}((\bar x,\bar y),\gph F_{\mathcal C})\right\}\; \forall y^*\in \mathbb R^m.$$
    In the case of $\mathcal C=\dom F,$ we use $\bar D^*F(\bar x,\bar y)$   instead of $D^*_{\mathcal C}F(\bar x,\bar y)$ and call the {\it relative limiting coderivative} of $F$ at $(\bar x,\bar y).$  }
    \end{definition}

    \begin{remark}
     {\rm (i) In Definition~\ref{coderivative-wrt}, if $\mathcal C$ is a neighborhood of $\bar x$, then  $D^*_{\mathcal C}F(\bar x,\bar y)$ reduces to $D^*F(\bar x,\bar y)$.}

     {\rm (ii)} In the case where $\mathcal C$ is a convex polyhedron, the limiting coderivative with respect to $\mathcal C$ in Definition~\ref{coderivative-wrt} coincides with that one in the normal contingent coderivative with respect to a set in \cite{MWY23}.
\end{remark}

Through the limiting coderivative with respect to a set, the necessary and sufficient conditions for the Aubin property with respect to a set of multifunctions were stated in \cite{TQ01-ZNU, TQ02-ZNU} as follows.

\begin{theorem}[A version with respect to a set of the Mordukhovich criterion] \label{thm1}
Let $\mathcal C$ be a closed and convex subset of $\mathbb R^n$. Let $F:\mathbb R^n\rightrightarrows\mathbb R^m$, $\bar x\in \mathcal C$, and $\bar y\in F(\bar x).$ Assume that $F$ has a locally closed graph around $(\bar x,\bar y)$. Then $F$ has the Aubin property with respect to $\mathcal C$ around $(\bar x,\bar y)$ if and only if
\begin{equation}\label{thm1-eqa}
D^*_{\mathcal C}F(\bar x,\bar y)(0)=\{0\}.
\end{equation}
\end{theorem}

We now introduce the concept of subdifferentials with respect to a set of single-valued mappings.

\begin{definition}[\cite{TQ02-ZNU}]
{\rm Consider the extended real-valued function $f:\mathbb R^n\to \bar\R.$ Let $\mathcal C$ be a closed and convex subset of $\mathbb R^n$, and let $\bar x\in\mathcal C.$

{\rm (i)} A vector $x^*\in\R^n$ satisfying $(x^*,-1)\in N_{\mathcal C\times \R}((\bar x,f(\bar x)),\epi f)$ is called a {\it limiting subgradient vector with respect to} $\mathcal C$ of $f$ at $\bar x.$  The set of all of limiting subgradient vectors with respect to $\mathcal C$ of $f$ at $\bar x$ is called the {\it limiting subdifferential with resepct to} $\mathcal C$ of $f$ at $\bar x.$ Thus
$$
\partial_{\mathcal C}f(\bar x):=\left\{x^*\in \mathbb R^n\mid (x^*,-1)\in N_{\mathcal C\times \R}((\bar x,f(\bar x)),\epi f)\right\}.
$$

{\rm (ii)} The {\it horizon subdifferential with resepct to} $\mathcal C$ of $f$ at $\bar x,$ denoted $\partial^{\infty}_{\mathcal C}f(\bar x)$, is defined by
$$
\partial^{\infty}_{\mathcal C}f(\bar x):=\left\{x^*\in \mathbb R^n\mid (x^*,0)\in N_{\mathcal C\times\R}((\bar x,f(\bar x)),\epi f)\right\}.
$$

If $\bar x\notin \dom f\cap\mathcal C,$ we put $\partial_{\mathcal C}f(\bar x):=\partial^{\infty}_{\mathcal C}f(\bar x):=\emptyset.$ In the case of $\mathcal C=\dom f,$ we write $\bar\partial f(\bar x), \bar\partial^{\infty}f(\bar x)$ instead of $\partial_{\mathcal C}f(\bar x),\partial^{\infty}_{\mathcal C}f(\bar x)$ and say it  the {\it relative limiting {\rm and} horizon subdifferentials}, respectively.}
\end{definition}

The necessary and sufficient condition for a singleton mapping satisfying the locally Lipschitz continuous property with respect to a set is presented in  the following theorem.

\begin{theorem}[\cite{TQ02-ZNU}]\label{thm2}
Let $f:\mathbb R^n\to \bar \R$, and let $\mathcal C$ be nonempty, closed, and  convex. Let $\bar x\in \mathcal C$ and $f\in \mathcal F(\bar x).$ Then $f$ is locally Lipschitz continuous with respect to $\mathcal C$ around $\bar x$ if and only if $\partial^{\infty}_{\mathcal C}f(\bar x)=\{0\}.$
\end{theorem}

For continuously process, we provide necessary optimality conditions due to the generalized differentials with respect to a set to the following mathematical program with equilibrium constraints:
\begin{align}\label{mpecs}
    \min \quad & f(x)\tag{MPEC}\\
    \text{such that } & 0\in G(x) \text{ \rm and } x\in \mathcal C_1\cap\mathcal C_2,\nonumber
\end{align}
where $f(x):\R^n\to\bar \R:=\R\cup\{\infty\}$, $G:\R^n\rightrightarrows\R^m$ and $\mathcal C_i\subset \R^n$ for $i=1,2$ are nonempty, convex, and closed  sets. We first recall the following definition.

\begin{definition}[\cite{TQ02-ZNU}, Definition~7~(ii)]
{\rm Let $\Omega_1,\Omega_2,\mathcal C_1$, and $\mathcal C_2$ be nonempty sets, and let $\bar x\in \Omega_1\cap\Omega_2\cap\mathcal C_1\cap\mathcal C_2.$ Denote $\Omega:=\Omega_1\cap\Omega_2$ and $\mathcal C:=\mathcal C_1\cap\mathcal C_2$. We say that $\{\Omega_1,\Omega_2\}$ are {\it normal-densed} in $\{\mathcal C_1,\mathcal C_2\}$ at $\bar x$ if whenever there are sequences $x_{ik}\xrightarrow{\Omega_i\cap\mathcal C_i}\bar x, x_k\xrightarrow{\Omega\cap\tbd \mathcal C}\bar x$ and $x_{ik}^*\in \hat N(x_{ik},\Omega_i\cap\mathcal C_i)$ for $i=1,2$, $x_k^*\in \mathcal R(x_k,\mathcal C)$ satisfying $x_{ik}^*\to x_i^*$ with some $x_i^*$ for $i=1,2$ and $x_k^*\to x_{1}^*+x_{2}^*,$ we find sequences $\tilde x_{ik}\xrightarrow{\Omega_i\cap\mathcal C_i}\bar x$ and $\tilde x_{ik}^*\in \hat N_{\mathcal C_i}(\tilde x_{ik}, \Omega_i)$ for $i=1,2$  such that $\tilde x_{ik}^*\to \tilde x_i^*$ with some $\tilde x_i^*$ for $i=1,2$ and
\begin{equation}\label{slqc-eq00}
\tilde x_{1}^*+\tilde x_{2}^*=x_1^*+x_2^*
\end{equation}
and $\max\{\Vert \tilde x_{1}^*\Vert, \Vert \tilde x_{2}^*\Vert\}>0 \text{ \rm whenever } \max\{\Vert x_{1}^*\Vert, \Vert x_{2}^*\Vert\} >0.$}
\end{definition}

\begin{theorem}[\cite{TQ02-ZNU}, Theorem~13]\label{nec-con} Consider the \ref{mpecs} problem. Let $\bar x$ be a local minimizer to \eqref{mpecs}, and let $f_{\mathcal C_1}\in \mathcal F(\bar x)$. Assume that the following qualification conditions are fulfilled:

{\rm($q_1$)} $\partial_{\mathcal C_1}^{\infty}f(\bar x)\cap [-D_{\mathcal C_2}^*G(\bar x,0)(0)]=\{0\}$.

{\rm($q_2$)} $\{\Omega_1,\Omega_2\}$ are normal-densed in $\{\mathcal C_1\times\R^{m+1},\mathcal C_2\times\R^{m+1}\}$ at $(\bar x,0,f(\bar x)),$ where $\Omega_i$ for $i=1,2$ is, respectively, defined by
\begin{equation}\label{omegai}
\Omega_1:=\{(x,y,z)\mid (x,z)\in \epi f, y\in \R^m\} \; \text{ \rm and } \Omega_2:=\gph G\times\R.\end{equation}

Then
\begin{equation}\label{thm14-eq0}
0\in \partial f_{\mathcal C_1}(\bar x)+D^*_{\mathcal C_2}G(\bar x,0)(0).
\end{equation}
\end{theorem}

\section{Normal Cones with respect to a Set of the Graphic of the Normal Cone Mapping}\label{normalcones}

This section is developed to establish normal cones with respect to a $\mathcal C\times\R^n$ of the normal cone mapping $\mathcal N(x):=N(x,\Theta),$ where $\mathcal C$ and $\Theta$ are defined by
\begin{equation}\label{Theta}
\Theta:=\left\{x\in \R^n\mid \langle a_i,x\rangle\le b_i\; \forall i\in I_1\right\}
\end{equation}
and
\begin{equation}\label{mathcalC}
\mathcal C:=\left\{x\in \R^n\mid \langle a_i,x\rangle\le 0\; \forall i\in I_2\right\},
\end{equation}
where $I_1:=\{1,\ldots, \ell_1\}$, $I_2:=\{\ell_1+1,\ldots, \ell\}$, $a_i\in \R^n$ and $b_i\in \R$ for each $i\in I_1\cup I_2.$ For each $x\in \Theta,$ and/or $x\in\mathcal C$, we define
\begin{equation}\label{I1}
    I_1(x):=\left\{i\in I_1\mid \langle a_i,x\rangle =b_i\right\}
\end{equation}
anhd/or
\begin{equation}\label{I2}
    I_2(x):=\left\{i\in I_2\mid \langle a_i,x\rangle=0\right\},
\end{equation}
respectively.
The set $I_1(x)$ (and/or $I_2(x)$) is called the {\it active index set} of $\Theta$ (and/or $\mathcal C$) at $x.$ It is known from \cite{Roc96} that the normal and tangent cones to $\Theta$ at $x\in \Theta$ are respectively computed by
\begin{align}\label{normalcone}
N(x,\Theta)&=\cone\{a_i\mid i\in I_1(x)\},\\
T(x,\Theta)&=\left\{u\in \R^n\mid \langle a_i, u\rangle\le 0\; \forall i\in I_1(x)\right\}.\label{tangentcone}
\end{align}
We now define
\begin{equation}\label{I}
I:=\{1, \ldots, \ell\} \; \text{ \rm and }\;
 I(x):=\left\{ i \in I \mid \langle a_i,x \rangle= 0 \right\}= I_1(x)\cup I_2(x).
\end{equation}
It is easy to see that
\begin{equation}\label{tangentThetaC}
T(x,\Theta\cap\mathcal C)=\left\{u\in \R^n\mid \langle a_i, u\rangle\le 0\;  \forall i\in I(x)\right\}.
\end{equation}
Let $\bar x\in\Theta$ and $\bar x^*\in N(\bar x,\Theta).$ There exist, by \eqref{normalcone}, a set $\mathcal I\subset I_1(\bar x)$ and real numbers $\lambda_i\ge 0$ for all $i\in \mathcal I$ such that $\bar x^*= \sum_{i\in \mathcal I}\lambda_ia_i.$
We define \begin{equation}\label{tildeI}\mathcal I(\bar x, \bar x^*):=
\{i\in \mathcal I  \mid \lambda_i>0\}.
\end{equation}
Let $P\subset Q\subset I$ with $P\subset I_1$. We now denote
$$\mathcal T_{Q}:=\left\{u\in \R^n\mid \langle a_i, u \rangle\le 0\; \forall i\in Q\right\},$$
$$
B_{Q,P}:=\left\{ u\in \R^n\mid u\in \mathcal T_{Q}, \langle a_i, u\rangle=0\; \forall i\in P \right\},
$$
and
$$
A_{Q,P}:=\cone\{a_i\mid i\in Q\setminus P\}+{\rm span}\{a_i\mid i\in P\}.
$$
It is known from \cite[Lemma 3.3]{HMN2010} that \begin{equation}\label{AB}
    B_{Q,P}^{\circ} = A_{Q,P}.
\end{equation}
The {\it normal cone mapping} to $\Theta$, denoted by $\mathcal N:\R^n\rightrightarrows\R^n$, is defined by
\begin{equation}\label{mathcalN}
\mathcal N(x):=\begin{cases} N(x,\Theta)& \text{\rm if } x\in \Theta;\\
\emptyset &\text{ \rm otherwise}.
\end{cases}
\end{equation}

Now it is the right time to provide the calculation formula for normal cones with respect to $\mathcal C\times\R^n$ of $\gph \mathcal N$. For convenience in comparing results, we will represent the calculation formulas of normal cones in a comparable form to those in \cite{HMN2010}. It should be noted that, in the proof of the inclusion ``$\subset$'' of the following theorem, we use the similar arguments as in  \cite[Theorem~3.4]{HMN2010}. However, in the proof of the opposite inclusion ``$\supset$'', instead of using the {\it reduction and absurdum} as in \cite[Theorem~3.4]{HMN2010}, we will use another way.

\begin{theorem}\label{FNCFormula}
Let $\Omega$ and $\mathcal C$ be respectively given by \eqref{Theta} and \eqref{mathcalC}. Let $\bar x\in \mathcal C$ and $(\bar x,\bar x^*)\in \gph \mathcal N$. Then the Fr\'{e}chet normal cone with respect to $\mathcal C\times\R^n$ of $\gph \mathcal N$ at $(\bar x,\bar x^*)$ is given by
\begin{align}\label{Fnormal-normalcone1}
\hat N_{\mathcal C\times\R^n}((\bar x,\bar x^*),\gph \mathcal N)&=\left(\big(T(\bar x,\Theta\cap\mathcal C)\cap\{\bar x^*\}^{\perp}\big)^{\circ}\cap T(\bar x,\mathcal C)\right) \times \left(T(\bar x,\Theta)\cap\{\bar x^*\}^{\perp}\right)\\
&=(A_{\bar I, \tilde I}\cap \mathcal T_{\bar I_2})\times B_{\bar I_1,\tilde I},\label{Fnormal-normalcone2}
\end{align}
where $\bar I:=I(\bar x), \bar I_i:=I_i(\bar x)$ for $i=1,2$ and $\tilde I:=\mathcal I(\bar x,\bar x^*)$ with $I(\bar x), I_1(\bar x), I_2(\bar x)$ and  $\mathcal I(\bar x,\bar x^*)$ given by \eqref{I}, \eqref{I1}, \eqref{I2} and \eqref{tildeI} respectively.
\end{theorem}

\begin{proof} First, we prove   relation \eqref{Fnormal-normalcone1}.
Taking $(u,v)\in \hat N_{\mathcal C\times\R^n}((\bar x,\bar x^*),\gph \mathcal N),$ we see from the definition that $(u,v)\in \hat N((\bar x,\bar x^*),\gph \mathcal N\cap\mathcal C\times\R^n)$ and $(u,v)\in T(\bar x,\mathcal C)\times\R^n.$ This implies that
\begin{equation}\label{thm19eq1}
\limsup_{(x,x^*)\xrightarrow{\tgph \mathcal N\cap\mathcal C\times\R^n}(\bar x,\bar x^*)}\dfrac{\langle (u,v), (x,x^*)-(\bar x,\bar x^*)\rangle}{\Vert (x,x^*)-(\bar x,\bar x^*)\Vert}\le 0.
\end{equation}
Picking $x=\bar x$ in \eqref{thm19eq1}, we obtain
\begin{equation*}\label{thm19eq2}
\limsup_{x^*\xrightarrow{\mathcal N(\bar x)}\bar x^*}\dfrac{\langle v, x^*-\bar x^*\rangle}{\Vert x^*-\bar x^*\Vert}\le 0
\end{equation*}
which gives us that
$$
v\in N(\bar x^*, \mathcal N(\bar x))=\big(\mathcal N(\bar x)\big)^{\circ}\cap\{\bar x^*\}^{\perp}=T(\bar x,\Theta)\cap \{\bar x^*\}^{\perp}.
$$
We next show that $u\in \big(T(\bar x,\Theta\cap\mathcal C)\cap \{\bar x^*\}^{\perp}\big)^{\circ}$, which is sufficient to prove \begin{equation}\label{thm19eq3}
\langle u,w\rangle\le 0\; \text{ \rm for all } w\in T(\bar x,\Theta\cap\mathcal C)\cap \{\bar x^*\}^{\perp}.
\end{equation}
To do this, we arbitrarily take $w\in T(\bar x,\Omega\cap\mathcal C)\cap \{\bar x^*\}^{\perp}$ and then we have $\langle a_i, w \rangle\le 0$ for all $i\in I(\bar x)$  and $\langle \bar x^*, w\rangle = 0.$ Construct the sequence $x_k:=\bar x+k^{-1}w$ for each $k\in\N.$ It is easy to see that
$$
\langle a_i, x_k\rangle = \langle a_i,\bar x \rangle+k^{-1}\langle a_i, w\rangle = k^{-1}\langle a_i, w\rangle\le 0, \; \forall i\in I(\bar x).
$$

On the other hand, since $\langle a_i, \bar x\rangle<0$ for any $i\in I\setminus I(\bar x)$,  we can choose $k_0\in \N$ such that the following inequalities hold:
$$
\langle a_i, \bar x\rangle\le k^{-1}\langle a_i,w\rangle, \; \forall i\in I\setminus I(\bar x) \text{ \rm and } k\ge k_0,
$$
which  implies that $x_k\in \mathcal C\cap \Theta$ for all $k\ge k_0.$ Furthermore, since $\bar x^*\in \mathcal N(\bar x)$ and $\langle \bar x^*, w\rangle = 0,$ for any $x\in \Theta,$ we have
$$
\langle \bar x^*, x-x_k \rangle = \langle \bar x^*, x-\bar x\rangle - k^{-1}\langle \bar x^*, w\rangle =\langle \bar x^*, x-\bar x\rangle\le 0.
$$
It follows that $\bar x^*\in N(x_k,\Theta)=\mathcal N(x_k)$, so $(x_k,\bar x^*)\in \gph \mathcal N\cap\mathcal C\times\R^n$ for all $k\ge k_0$. Taking $(x,x^*):=(x_k,\bar x^*)$ in \eqref{thm19eq1}, we have
$$
\limsup_{k\to\infty}\dfrac{\langle u,x_k-\bar x\rangle}{\Vert x_k-\bar x\Vert}\le 0.
$$
By the definition of $x_k$, we obtain
$$
\limsup_{k\to\infty}\dfrac{\langle u, w\rangle}{\Vert w\Vert}\le 0,
$$
which means that $\langle u, w\rangle\le 0$. So we achieve \eqref{thm19eq3} and thus the inclusion ``$\subset$'' in \eqref{Fnormal-normalcone1} holds.

We next prove the opposite inclusion \eqref{Fnormal-normalcone1} holds also. We arbitrarily take
$$
(u,v)\in \left(\big(T(\bar x,\Theta\cap\mathcal C)\cap\{\bar x^*\}^{\perp}\big)^{\circ}\cap T(\bar x,\mathcal C)\right) \times \left(T(\bar x,\Theta)\cap\{\bar x^*\}^{\perp}\right).
$$
For any $(x,x^*)\in \gph \mathcal N\cap(\mathcal C\times\R^n)$ with $x$ closed to $\bar x$ enough, we have $z\in \Theta\cap\mathcal C$ and $x^*\in N(x,\Theta).$ Since $\Theta$ is a convex polyhedron, we can assume that  $N(x,\Theta)\subset N(\bar x,\Theta).$ Moreover, from the fact that $\Theta\cap\mathcal C$ is a convex polyhedron, $N(\bar x,\Theta\cap\mathcal C)$ is also. Therefore,
\begin{align*}
T(\bar x,\Theta\cap\mathcal C)\cap\{\bar x^*\}^{\perp}&=(N(\bar x,\Theta\cap\mathcal C))^{\circ}\cap\{\bar x^*\}^{\perp}\\
&=N(\bar x^*,N(\bar x,\Theta\cap\mathcal C))
\end{align*}
which implies that
\begin{align*}
(T(\bar x,\Theta\cap\mathcal C)\cap\{\bar x^*\}^{\perp})^{\circ}&= T(\bar x^*,N(\bar x,\Theta\cap\mathcal C))\\
&=\{t(\tilde x^*-\bar x^*)\mid \tilde x^*\in N(\bar x,\Theta\cap\mathcal C)\}.
\end{align*}
To prove $(u,v)\in \hat N((\bar x,\bar x^*),\gph \mathcal N\cap(\mathcal C\times\R^n)),$ it is sufficient to show that
\begin{equation}
\label{thm-eq00}
\langle(u,v), (w,z)\rangle\le 0\; \forall (w,z)\in T((\bar x,\bar x^*),\gph \mathcal N\times(\mathcal C\times\R^n)).
\end{equation}
Pick  $(w,z)\in T((\bar x,\bar x^*),\gph \mathcal N\times(\mathcal C\times\R^n)).$ By the definition, we find sequences $t_k\downarrow 0, (w_k,z_k)\to (w,z)$ such that
\begin{equation}\label{thm-eq01}
\begin{cases}
\bar x+t_kw_k\in\mathcal C\cap\Theta\\
\bar x^*+t_kz_k\in N(\bar x+t_kw_k,\Theta).
\end{cases}
\end{equation}
Since $u\in \left(T(\bar x,\Theta\cap\mathcal C)\cap\{\bar x^*\}^{\perp}\right)^{\circ},$ there exist $\lambda\ge 0$ and $\tilde x^*\in N(\bar x,\Theta\cap\mathcal C)$ such that $u=\lambda(\tilde x^*- \bar x^*).$ If $\lambda=0$, then $u=0$. It follows that  $\langle u,w_k\rangle=0.$ If $\lambda>0$, then $\bar x^*+\frac{1}{\lambda}u\in N(\bar x,\Theta\cap\mathcal C)$, which yields that
$$
0\ge \langle \bar x^*+\frac{1}{\lambda}u, \bar x+t_kw_k-\bar x\rangle = t_k\langle\bar x^*,w_k\rangle+\frac{t_k}{\lambda}\langle u,w_k \rangle.
$$
Thus
\begin{equation}
\label{thm-eq02}
\langle u,w_k\rangle\le -\lambda\langle \bar x^*, w_k\rangle.
\end{equation}
On the other hand, from \eqref{thm-eq01}, we can see that $\langle \bar x^*+t_kz_k,\bar x-\bar x-t_kw_k \rangle \le 0,$
which is equivalent to $-\langle \bar x^*, w_k\rangle\le -t_k\langle z_k,w_k\rangle.$ Taking \eqref{thm-eq02} into account and passing to the limit as $k\to\infty$, we achieve $\langle u,w \rangle\le 0.$ It remains to prove $\langle v,z\rangle\le 0.$ Since $v\in T(\bar x,\Theta)\cap\{\bar x^*\}^{\perp}$ and $\bar x^*+t_kz_k\in N(\bar x+t_kw_k,\Theta)\subset N(\bar x,\Theta)$ for $k$ large enough, we have
$$
0\ge \langle \bar x^*+t_kz_k, v\rangle = \langle\bar x^*, v\rangle+t_k\langle v,z_k\rangle=t_k\langle v,z_k\rangle,
$$
which  implies that $\langle v,z_k\rangle\le 0$. Letting $k\to \infty$, we obtain $\langle v,z\rangle\le 0.$ Thus,  inequality \eqref{thm-eq00} holds, so $(u,v)\in \hat N((\bar x,\bar x^*),\gph \mathcal N\cap(\mathcal C\times\R^n)).$
Moreover, it implies from $u\in T(\bar x,\mathcal C)$ and the convex polyhedral property of $\mathcal C$ that
$$
(u,v)\in T((\bar x,\bar x^*),\mathcal C\times\R^n)=\mathcal R((\bar x,\bar x^*),\mathcal C\times \R^n),
$$
which derives that $(u,v)\in \hat N_{\mathcal C\times\R^n}((\bar x,\bar x^*),\gph \mathcal N).$
This tells us that   opposite inclusion \eqref{Fnormal-normalcone1} holds.

Next, we prove \eqref{Fnormal-normalcone2}. Indeed, it is easy to see that
$$
T(\bar x,\Theta)\cap\{\bar x^*\}^{\perp} = \left\{ v\in \R^n\mid \langle a_i, v\rangle\le 0\; \forall i\in \bar I_1\setminus\tilde I \text{ \rm and } \langle a_i, v\rangle =0\; \forall i\in \tilde I\right\}=B_{\bar I_1,\mathcal I},
$$
while by the similar arguments, we obtain
$$
T(\bar x,\Theta\cap\mathcal C)\cap\{\bar x^*\}^{\perp} = \left\{ v\in \R^n\mid \langle a_i, v\rangle\le 0\; \forall i\in \bar I\setminus\tilde I \text{ \rm and } \langle a_i, v\rangle =0\; \forall i\in \tilde I\right\}=B_{\bar I,\mathcal I}.
$$
From \eqref{AB}, we assert that $(T(\bar x,\Theta\cap\mathcal C)\cap\{\bar x^*\}^{\perp})^{\circ} = A_{\bar I, \mathcal I}.$ Taking \eqref{Fnormal-normalcone1} and the fact that $\mathcal T_{\bar I_2}=\mathcal R(\bar x,\mathcal C)$ into account, we obtain  relation~\eqref{Fnormal-normalcone2} immediately.
\end{proof}

In what follow, we   establish the formula for the limiting normal cone with respect to a set of the graphic of normal cone mapping to $\Theta$. To do this, we need to present some notations as follows:
let $Q\subset I$ and $(\bar x,\bar x^*)\in \gph \mathcal N$ with $\bar x\in \mathcal C.$ We denote
\begin{equation} \label{Qi}
    Q\vert_{I_j}:=\left\{ i\in Q\mid i\in I_j\right\}\; \text{ \rm for } j=1,2,
\end{equation}
\begin{equation}\label{Cq}
C_Q:=\left\{x\in\mathbb R^n\mid \langle a_i,x\rangle=0\; \forall i\in Q, \text{ \rm and } \langle a_i,x\rangle<0\; \forall i\in I\setminus Q\right\},
\end{equation}
\begin{equation}\label{barCq}
\bar C_Q:=\left\{x\in\mathbb R^n\mid \langle a_i,x\rangle=0\; \forall i\in Q, \text{ \rm and } \langle a_i,x\rangle\le 0\; \forall i\in I\setminus Q\right\},
\end{equation}
and
\begin{equation}\label{mathcalP}
    \mathcal P(\bar x,\bar x^*):=\left\{ P\subset I_1(\bar x)\mid \bar x^*\in \cone\{a_i\mid i\in P\} \right\}.
\end{equation}

\begin{theorem}\label{LNCFormula}
Let $\Omega$ and $\mathcal C$ be respectively given by \eqref{Theta} and \eqref{mathcalC}. Let $\bar x\in \mathcal C$ and $(\bar x,\bar x^*)\in \gph \mathcal N$. Then the limiting normal cone with respect to $\mathcal C\times\R^n$ of $\gph \mathcal N$ at $(\bar x,\bar x^*)$ is given by
\begin{align}\label{limiting-normalcone}
N_{\mathcal C\times\R^n}((\bar x,\bar x^*),\gph \mathcal N)
&=\bigcup_{\substack{Q\subset \bar I, C_{Q}\ne\emptyset\\ Q\vert_{I_1}\supset P\in \mathcal P}}(A_{Q, P}\cap \mathcal T_{Q\vert_{I_2}})\times B_{Q\vert_{I_1},P},
\end{align}
where $Q\vert_{I_i},$ for $i=1,2$, $C_Q$ ,and $\mathcal P:=\mathcal P(\bar x,\bar x^*)$ are given by \eqref{Qi}, \eqref{Cq}, and \eqref{mathcalP} respectively, and $\bar I:=I(\bar x)$ with $I(\bar x)$ as in \eqref{I}.
\end{theorem}

\begin{proof}
We first demonstrate the inclusion ``$\subset$''. Taking $(u,v)\in N_{\mathcal C\times\R^n}((\bar x,\bar x^*),\gph \mathcal N),$ one sees that there exist by the definition sequences $(x_k,x_k^*)\xrightarrow{\tgph \mathcal N\cap\mathcal C\times\R^n}(\bar x,\bar x^*)$ and $(u_k,v_k)\to (u,v)$ such that
$$
(u_k,v_k)\in \hat N_{\mathcal C\times\R^n}((x_k,x_k^*),\gph \mathcal N)\; \text{ \rm for all } k\in \N.
$$
It follows that $x_k\in \Theta\cap\mathcal C$ and $x_k^*\in N(x_k,\Theta)$ for all $k\in \N.$ Since $x_k\to \bar x$ and $I$ is finite, we can assume that $I(x_k)\subset I(\bar x)$ for all $k\in\N.$ By the finiteness of $I(\bar x)$, we find that there exist a subset $Q$ of $I(\bar x)$ and a subsequence $(x_{k_s})$ of $(x_k)$ such that $I(x_{k_s}) = Q$ for all $x_{k_s}\in \N.$ This implies that
$$
\langle a_i, x_{k_s}\rangle=0\; \forall i\in Q \text{ \rm and } \langle a_i, x_{ks}\rangle<0\; \forall i\in I\setminus Q,
$$
which gives us that $C_Q\ne \emptyset.$ Using \eqref{normalcone}, we find for each $x_{k_s}^*$ numbers $\lambda_{ik_s}\ge 0$ for all $i\in Q\vert_{I_1}$ such that $x_{k_s}^*=\sum_{i\in Q\vert_{I_1}} \lambda_{ik_s}a_i.$ For each $x_{k_s}^*,$ we define the set
$$
P(x_{k_s}^*):=\left\{i\in Q\vert_{I_1}\mid\lambda_{ik_s}>0\right\}.
$$
Then $P(x_{k_s}^*)$ is finite for any $k_s$. Moreover,   $Q\vert_{I_i}=Q\cap I_i=I(x_k)\cap I_i=I_i(x_k)$ for $i=1,2.$ Without loss of the generality (take a subsequence if necessary), we can assume that $P(x_{k_s})=P$ for some $P\subset Q\vert_{I_1}=I_1(x_{k_s})$ for all $x_{k_s}.$ It is not difficult to see that
$$
x_{k_s}^*=\sum_{i\in P}\lambda_{ik_s}a_i\in \cone\{a_i\mid i\in P\},
$$
so $\bar x^*\in \cone\{a_i\mid i\in P\}$ due to the closedness of  finitely generated cones and the fact that $x_{k_s}^*\to \bar x^*$ as $s\to\infty.$ This follows that $P\in\mathcal P$, which implies that $Q\vert_{I_1}\supset P\in \mathcal P$. We now use the relation~\eqref{Fnormal-normalcone2} in Theorem~\ref{FNCFormula} to obtain
\begin{align*}
(u_{k_s},v_{k_s})\in \hat N_{\mathcal C\times\R^n}((x_{k_s},x_{k_s}^*),\gph \mathcal N)&=A_{I(x_{k_s}), \mathcal I(x_{k_s},x_{k_s}^*)}\cap\mathcal T_{I_2(x_{k_s})}\times B_{I_1(x_{k_s}),\mathcal I(x_{k_s},x_{k_s}^*)}
\\
&= A_{Q,P}\cap \mathcal T_{Q\vert_{I_2}}\times B_{Q\vert_{I_1},P}.
\end{align*}
Passing to the limit as $s\to \infty$ and using the closedness of $A_{Q,P}, \mathcal T_{Q\vert_{I_2}}, B_{Q\vert_{I_1},P},$ we have
$$
(u,v)\in A_{Q,P}\cap \mathcal T_{Q\vert_{I_2}}\times B_{Q\vert_{I_1},P}
$$
and hence the inclusion ``$\subset$'' in \eqref{limiting-normalcone} holds.

We next prove the opposite inclusion in \eqref{limiting-normalcone}. Let $u\in A_{Q,P}\cap \mathcal T_{Q\vert_{I_2}}$ and $v\in B_{Q\vert_{I_1},P}$ with $Q\subset \bar I$, $C_Q\ne \emptyset,$ $Q\vert_{I_1}\supset P\in \mathcal P.$  Take $\tilde x\in C_Q$ and construct a sequence
\begin{equation*}\label{xk}
    x_k:=k^{-1}\tilde x+(1-k^{-1})\bar x\to \bar x \text{ \rm as } k\to \infty.
\end{equation*}
For each $i\in Q$ and $j\in I\setminus Q,$ we have
\begin{equation}\label{xkpro}
    \langle a_i,x_k\rangle =  k^{-1}\langle a_i, \tilde x \rangle+(1-k^{-1})\langle a_i, \bar x\rangle= 0
\end{equation}
due to $\tilde x\in C_Q$, and $\langle a_j, x_k \rangle =  k^{-1}\langle a_j, \tilde x \rangle+(1-k^{-1})\langle a_j, \bar x\rangle< 0,$
which implies that $x_k\in \Theta\cap\mathcal C$ and $Q\vert_{I_1}=I_1(x_k)$ for all $k\in \N.$ Using \eqref{normalcone}, we derive
\begin{equation*}\label{normalcone2}
    N(x_k,\Theta)=\cone\left\{ a_i\mid i\in Q\vert_{I_1}\right\}\; \forall k\in \N.
\end{equation*}
Since $P\in \mathcal P$, there exist $\lambda_i\ge 0$ such that  $\bar x^*=\sum_{i\in P}\lambda_i a_i.$  For each $k\in \N,$ we construct the sequence $x_k^*:=\sum_{i\in P}(\lambda_i+k^{-1})a_i.$ Then $x_k^*\to x^*$ and $x_k^*\in N(x_k,\Theta)$ for all $k\in \N.$ Moreover, we have, from the definitions of $P$ and $x_k^*,$ that $P=\mathcal I(x_k,x_k^*)$ for all $k\in\N$ with $\mathcal I(x_k,x_k^*)$ defined by \eqref{tildeI}. On the other hand, it is easy to see from \eqref{xkpro} that $Q\vert_{I_2}= I_2(x_k).$ By the facts that  $(u,v)\in A_{Q,P}\cap \mathcal T_{Q\vert_{I_2}}\times B_{Q\vert_{I_1}, P}$, we obtain
$$
(u,v)\in A_{I(x_k),\mathcal I(x_k,x_k^*)}\cap \mathcal T_{I_{2}(x_k)}\times B_{I_1(x_k),\mathcal I(x_k,x_k^*)},
$$
which gives us that $(u,v)\in \hat N_{\mathcal C\times\R^n}((x_k,x_k^*),\gph \mathcal N)$ for all $k\in \N$, which is due to Theorem~\ref{FNCFormula}. Passing to the limit as $k\to \infty,$ we have $(u,v)\in N_{\mathcal C\times \R^n}((\bar x,\bar x^*),\gph \mathcal N).$
Therefore,
$$
A_{Q,P}\cap \mathcal T_{Q\vert_{I_2}}\times B_{Q\vert_{I_1}, P}\subset N_{\mathcal C\times \R^n}((\bar x,\bar x^*),\gph \mathcal N).
$$
This demonstrates that the inclusion ``$\supset$'' in \eqref{limiting-normalcone} is fulfilled.
Hence we obtain  relation \eqref{limiting-normalcone} and the proof of this theorem is completed.
\end{proof}

In the case that  the generating system $\{a_i\mid i\in I(\bar x)\}$ is linear independent, the simple formula of the limiting normal cone with respect to $\mathcal C\times\R^n$ of $\gph \mathcal N$ at $(\bar x,\bar x^*)$ is provided in the following theorem.

\begin{theorem}\label{LNCFormula1}
Let $\bar x\in \mathcal C$, and let $(\bar x,\bar x^*)\in \gph \mathcal N$ in the framework of Theorem~\ref{LNCFormula}. Assume that the generating elements $\{a_i\mid i\in I(\bar x)\}$ are linear independent.
Then the limiting normal cone with respect to $\mathcal C\times\R^n$ of $\gph \mathcal N$ at $(\bar x,\bar x^*)$ is computed by
\begin{align}\label{limiting-normalcone1}
N_{\mathcal C\times\R^n}((\bar x,\bar x^*),\gph \mathcal N)
&=\bigcup_{\substack{Q\subset \bar I\\
\tilde I\subset P\subset Q\vert_{I_1}}}(A_{Q, P}\cap \mathcal T_{Q\vert_{I_2}})\times B_{Q\vert_{I_1},P},
\end{align}
where $\tilde I:=\mathcal I(\bar x,\bar x^*)$ is given by \eqref{tildeI}.
\end{theorem}

\begin{proof}
For any $Q\subset I(\bar x),$ we consider the following equations system
\begin{equation}\label{eqs}
\langle a_i, x\rangle=0\; \text{ \rm for all } i\in Q \; \text{ \rm and } \langle a_i,x \rangle=-1\; \text{ \rm for all }  i\in I(\bar x)\setminus Q.
\end{equation}
Since $\{a_i\mid i\in I(\bar x)\}$ are linear dependent,  system~\eqref{eqs} has a solution $\tilde x$, which satisfies $\bar x+t\tilde x\in \Theta$ and $I(\bar x+t\tilde x)=Q$ for  $t>0$ sufficiently small. Thus we have $\bar x+t\tilde x\in C_Q$ and then $C_Q\ne \emptyset.$ To obtain \eqref{limiting-normalcone1}, it remains to demonstrate that
\begin{equation}\label{Pcond}
P\in \mathcal P(\bar x,\bar x^*) \Leftrightarrow \tilde I\subset P.
\end{equation}
If $\bar x^*=0$, then $\tilde I =\emptyset$, so \eqref{Pcond} holds. We consider the case that $\bar x^*\ne 0.$ It is easy to see that the implication ``$\Leftarrow$'' trivially holds. We only need to prove the necessary condition of \eqref{Pcond}. Indeed, let $P\in \mathcal P(\bar x,\bar x^*).$ We find $\lambda_i\ge 0$ for all $i\in P$ such that $\bar x^*=\sum_{i\in P}\lambda_ia_i.$

On the other hand, there exist by \eqref{tildeI} $\nu_j>0$ for all $j\in \tilde I$ satisfying $\bar x^*=\sum_{j\in \tilde I}\nu_ja_j.$ We now define nonnegative numbers by
$$
\alpha_i:=\begin{cases}
\lambda_i& \textbf{ \rm if } i\in P\\
0&\text{ \rm if } i\in \bar I\setminus P,
\end{cases} \quad \text{ \rm and } \beta_i:=\begin{cases}
\nu_i& \textbf{ \rm if } i\in \tilde I\\
0&\text{ \rm if } i\in \bar I\setminus \tilde I.
\end{cases}
$$
Then we have $\bar x^*=\sum_{i\in \bar I}\alpha_ia_i = \sum_{i\in \bar I}\beta_ia_i.$
Since $\{a_i\mid i\in \bar I \}$ are linear independent, we achieve $\alpha_i=\beta_i$  for all $i\in \bar I,$ which means that
$$
\lambda_i= \nu_i\; \text{ \rm for all } i\in \tilde I, \text{ \rm and } \lambda_i=0\; \text{ \rm for all } i\in \bar I\setminus \tilde I.
$$
For any $i\in \tilde I,$ we obtain  $\lambda_i=\nu_i>0$, which gives us that $i\in P.$ Thus $\tilde I\subset P$
and relation \eqref{Pcond} holds.
\end{proof}

\begin{remark}
{\rm If $\bar x\in {\rm int}\,\mathcal C$, then $I_2(\bar x)=\emptyset$, which implies that $\mathcal T_{Q\vert_{I_2}} = (\cone\{\emptyset\})^{\circ} = \R^n.$ Thus Theorems~\ref{FNCFormula}, \ref{LNCFormula}, and \ref{LNCFormula1} reduce to \cite[Theorems~3.4, 4.1 and 4.2]{HMN2010},  respectively.}
\end{remark}

\section{Formulas of the corderivatives with respect to a set of the normal cone mapping}\label{coderivative}
In this section, by using the obtained formulas of the limiting normal cone with respect to $\mathcal C\times\R^n$ of $\gph \mathcal N$ in Section~\ref{normalcones}, we provide formulas of the limiting coderivative with respect to $\mathcal C$ of the normal cone mapping $\mathcal N.$ The following theorem is directly proved from Definition~\ref{coderivative-wrt}, Theorems~\ref{LNCFormula}, and~\ref{LNCFormula1}. For simplicity, we denote $\bar I:=I(\bar x)$ and $\bar I_i:=I_i(\bar x)$ for $i=1,2$ and $\bar x\in \Theta.$

\begin{theorem}\label{LCFormula1}
Let $\bar x\in \mathcal C$, and let $(\bar x,\bar x^*)\in \gph \mathcal N$ in the framework of Theorem~\ref{LNCFormula}. Then the limiting coderivative with respect to $\mathcal C$ of $\gph \mathcal N$ at $(\bar x,\bar x^*)$ is given by
\begin{align}\label{Lcoderivative-normalcone}
D^*_{\mathcal C}\mathcal N(\bar x,\bar x^*)(u) = \Big\{v\mid  (v,-u)\in &\,  A_{Q, P}\cap \mathcal T_{Q\vert_{I_2}}\times B_{Q\vert_{I_1},P}  \text{ \rm for some } Q\subset\bar I, P\in \mathcal P(\bar x,\bar x^*)\nonumber\\
& \text{ \rm with } C_Q\ne \emptyset, P\subset Q\vert_{I_1}\Big\}.
\end{align}
If, in addition, the generating elements $\{a_i\mid i\in I(\bar x)\}$ are linear independent, then
\begin{align}\label{Lcoderivative-normalcone1}
D^*_{\mathcal C}\mathcal N(\bar x,\bar x^*)(u) = \Big\{v\mid  (v,-u)\in &\,  A_{Q, P}\cap \mathcal T_{Q\vert_{I_2}}\times B_{Q\vert_{I_1},P}  \text{ \rm for some } Q\subset\bar I,\nonumber\\
& \text{ \rm and } \mathcal I(\bar x,\bar x^*)\subset P\subset Q\vert_{I_1}\Big\}.
\end{align}
\end{theorem}

\begin{remark}
{\rm Theorem~\ref{LCFormula1} provides a computable formula of the limiting coderivative with respect to a set of the normal cone mapping $\mathcal N.$ However, formulas \eqref{Lcoderivative-normalcone} and \eqref{Lcoderivative-normalcone1} may not be so favorable to calculate. In what follow, we   establish more favorable calculation formulas.}
\end{remark}

Given $u\in \R^n$, we define the {\it characteristic active index subsets}
\begin{equation}\label{Iu}
    I_1(\bar x, u):=\left\{i\in I_1(\bar x)\mid \langle a_i, u\rangle = 0 \right\}\; \text{ \rm and } \; I_1^+(\bar x, u):=\left\{i\in I_1(\bar x)\mid \langle a_i, u\rangle>0\right\}.
\end{equation}
Let $T\subset I(\bar x).$ We denote
\begin{align}
    \Upsilon(T):=\big\{i\in I_1(\bar x)\mid \langle a_i,x\rangle =0 \text{ \rm for all } x \in \bar C_T \big\}.
    \label{upsilonT}
\end{align}

\begin{theorem}\label{DCNTheorem}
Let $\bar x\in \mathcal C$, and let $(\bar x,\bar x^*)\in \gph \mathcal N$ in the framework of Theorem~\ref{LNCFormula}. Give $u\in \R^n$ and let $I_1(\bar x, u), I_1^+(\bar x, u)$ be defined by \eqref{Iu}.  The following assertions hold:

    {\rm (i)} ${\rm Dom}\, D^*_{\mathcal C}\mathcal N(\bar x,\bar x^*)$ can be computed by
    \begin{align}
   \label{domD} {\rm Dom}\, D^*_{\mathcal C}\mathcal N(\bar x,\bar x^*)&=\left\{u\in \R^n\mid \langle a_i,u\rangle=0\; \forall i\in \tilde I,\; \langle a_i,u\rangle\ge 0\;\forall j\in \Upsilon(\tilde I)\setminus\tilde I\right\},\end{align} where $\tilde I:=\mathcal I(\bar x,\bar x^*)$ is defined by \eqref{tildeI}. Moreover, for all $u\in {\rm Dom}\, D^*_{\mathcal C}\mathcal N(\bar x,\bar x^*)$, we have
   \begin{align}
  \label{DCN}  D^*_{\mathcal C}\mathcal N(\bar x,\bar x^*)(u)&\subset \left(\cone\{a_i\mid i\in I_1^+(\bar x, u)\cup \bar I_2\} + {\rm span}\{a_i\mid i\in I_1(\bar x, u)\}\right)\cap\left(\bigcup\limits_{T\in \mathcal I_2}\mathcal T_{T}\right),
    \end{align}  where $\mathcal I_2:=\left\{T\subset \bar I_2\mid C_{I_1(\bar x,u)\cup I_1^+(\bar x,u)\cup T}\ne \emptyset\right\}.$

 {\rm (ii)}  If, in addition, the generating elements $\{a_i\mid i\in I(\bar x)\}$ are linear independent, then
    \begin{align}\label{DCNnew}
    &D^*_{\mathcal C}\mathcal N(\bar x,\bar x^*)(u)= \left(\cone\{a_i\mid i\in I_1^+(\bar x, u)\cup \bar I_2\} + {\rm span}\{a_i\mid i\in I_1(\bar x, u)\}\right)\cap\left(\bigcup\limits_{T\in \mathcal I_2}\mathcal T_{T}\right)
    \end{align}  for all $u\in {\rm Dom}\, D^*_{\mathcal C}\mathcal N(\bar x,\bar x^*)$, where ${\rm Dom}\, D^*_{\mathcal C}\mathcal N(\bar x,\bar x^*)$ is given by \eqref{domD}.
\end{theorem}

\begin{proof}
We first prove  assertion (i). Take $(u,v)\in \gph D^*_{\mathcal C}\mathcal N(\bar x,\bar x^*).$ By the definition, we have $(v,-u)\in N_{\mathcal C}((\bar x,\bar x^*),\gph \mathcal N),$ which is equivalent to, due to Theorem~\ref{LNCFormula1}, that there exist $Q\subset I(\bar x)$ and $ P\subset Q$ such that $C_Q\ne \emptyset$, $P\in \mathcal P:=\mathcal P(\bar x,\bar x^*)$, and   $(v,-u)\in A_{Q,P}\cap\mathcal T_{Q\vert_{I_2}}\times B_{Q\vert_{I_1},P}.$  This means that
  \begin{align}
 \label{DCNTheorem-eq1}         v\in \left(\cone\{a_i\mid i\in Q\setminus P\}+{\rm span}\{a_i\mid i\in P\}\right) \cap\mathcal T_{Q\vert_{I_2}},\\
  \label{DCNTheorem-eq2}        \langle a_i, u\rangle\ge 0\; \forall i\in Q\vert_{I_1} \text{ \rm and } \langle a_i, u\rangle = 0\; \forall i\in P.
\end{align}
It follows that \eqref{DCNTheorem-eq2} is necessary and sufficient to $u\in {\rm Dom}\, D^*_{\mathcal C}\mathcal N(\bar x,\bar x^*).$
Let $u$ satisfy \eqref{DCNTheorem-eq2}. We next prove that
\begin{equation}\label{DCNTheorem-eq5}
\langle a_i, u\rangle= 0\; \forall i\in \tilde I \text{ \rm and } \langle a_i, u\rangle\ge 0\; \forall i\in \Upsilon(\tilde I)\setminus \tilde I,
\end{equation}
 with
$\tilde I=\mathcal I(\bar x,\bar x^*).$ We first show that the following equations \begin{equation}
     \label{DCNTheorem-eq3}
     \langle a_i, x\rangle = 0,\; \forall i\in \tilde I
 \end{equation}
hold for any $x\in C_Q$, where $C_Q$ is defined by \eqref{Cq}. Indeed, according to $P\in\mathcal P,$ we can assume that $\bar x^*=\sum_{j\in P}\mu_ja_j$ with  $\mu_j\ge 0$ for all $j\in P.$ Pick an arbitrarily $x\in C_Q.$ Since $P\subset Q$, we have $\langle \bar x^*, x\rangle = 0$, which is  due to the definition of $C_Q.$ From the definition of $\mathcal I(\bar x,\bar x^*),$ we achieve
$\langle \sum_{i\in \tilde I}\lambda_ia_i, x\rangle=\langle \bar x^*, x\rangle = 0,$ where $\lambda_i> 0$ for all $i\in \tilde I$ and $\bar x^*=\sum_{i\in\tilde I}\lambda_ia_i.$ This together with the fact that $\langle a_i, x \rangle \le 0$ for all $ i\in \tilde I$ yields  \eqref{DCNTheorem-eq3}. Therefore, $\tilde I\subset Q$ and then $\tilde I\subset Q\vert_{I_1}$ due to  $\tilde I\subset I_1.$ Taking \eqref{DCNTheorem-eq2} into account, we have $\langle a_i, u\rangle\ge 0$ for all $i\in \tilde I$. Combining this and the fact that $\langle \bar x^*, u\rangle=0$, we obtain
\begin{equation}
\langle a_i,u\rangle=0\; \forall i\in \tilde I.\label{DCNTheorem-eq4}
\end{equation}
We next prove that $\Upsilon(\tilde I)\subset Q\vert_{I_1}.$ Take arbitrarily $i_0\in \Upsilon(\tilde I)\subset I_1(\bar x).$ From the nonemptiness of $C_Q$, there is $x\in \R^n$ such that $\langle a_i, x\rangle = 0$ for all $i\in Q$ and $\langle a_i, x\rangle<0$ for all $i\in I\setminus Q$. It follows that $\langle a_{i_0},x\rangle=0$, which gives us that $i_0\in Q$ according to the definition of $C_Q$ and the fact that $x\in C_Q.$ Therefore, $\Upsilon(\tilde I)\subset Q$, which implies that $\Upsilon(\tilde I)\subset Q\vert_{I_1}$. In view of  \eqref{DCNTheorem-eq2},  one has $\langle a_i, u\rangle\ge 0$  for all $i\in \Upsilon(\tilde I)\setminus \tilde I$. Combining this and \eqref{DCNTheorem-eq4}, we derive \eqref{DCNTheorem-eq5}.

To complete the proof of \eqref{domD}, it remains to show that the relations in \eqref{DCNTheorem-eq5} imply $u\in {\rm Dom}\, D^*_{\mathcal C}N(\bar x,\bar x^*).$ Indeed, let \eqref{DCNTheorem-eq5} hold. Putting $P:=\tilde I$ and $Q:=\Upsilon(\tilde I)$, one has $Q=Q\vert_{I_1}=\Upsilon(\tilde I).$ For each $i\in I(\bar x)\setminus \Upsilon(\tilde I)$, one finds that there exists, by the definition of $\Upsilon(\tilde I)$ (see \eqref{upsilonT}), an element $x_i$ satisfying
$$
\langle a_j, x_i\rangle= 0\; \forall j\in P, \; \langle a_j,x_i\rangle\le 0\; \forall j\in I \setminus P \; \text{ \rm and } \langle a_i, x_i\rangle<0.
$$
For each $i\in I\setminus I(\bar x),$ we put $x_i:=\bar x$. Hence, $\langle a_j, x_i\rangle= 0$ for all $j\in I(\bar x)$ and  $\langle a_i, x_i\rangle<0.$ By the expression $I\setminus Q=I\setminus \Upsilon(\tilde I)=\big(I\setminus I(\bar x)\big)\cup \big(I(\bar x)\setminus \Upsilon(\tilde I)\big),$ we have the following assertion:
for each $i\in I\setminus Q$, there is $x_i\in \R^n$ satisfying
$$\langle a_j, x_i\rangle= 0,\; \forall j\in P, \; \langle a_j,x_i\rangle\le 0,\; \forall j\in I\setminus P \; \text{ \rm and } \langle a_i, x_i\rangle<0.$$
Setting $x:=\sum_{i\in I\setminus Q\vert_{I_1}}x_i,$ one sees that the following relations hold:
$$
\langle a_i, x \rangle=0\; \forall i\in P \text{ \rm and } \langle a_i,x\rangle\le 0\; \forall i\in I\setminus P.
$$
Therefore, $x\in \bar C_{P}$. From the definition $\Upsilon(\tilde I)$,  $P=\tilde I,$ and $Q=Q\vert_{I_1}=\Upsilon(\tilde I)$, one has
$$
\langle a_i, x\rangle= 0\; \forall i\in Q\; \text{ \rm and } \langle a_i, x\rangle= \langle a_i, x_i\rangle+\sum_{i\ne j\in I\setminus Q}\langle a_i,x_j\rangle<0\; \forall i\in I\setminus Q.
$$
It follows that $x\in C_Q$ and thus $C_{Q}\ne \emptyset.$ So we have $\bar I\supset Q=Q\vert_{I_1}\supset P\in\mathcal P$ and $C_Q\ne\emptyset.$ Moreover, it easy to see (even if $Q\vert_{I_2}=\emptyset$) that
$$
(0,-u)\in \big(A_{Q,P}\cap \mathcal T_{Q\vert_{I_2}}\big)\times B_{Q\vert_{I_1}, P}
$$
and hence $u\in {\rm Dom}\, D^*_{\mathcal C}\mathcal N(\bar x,\bar x^*)$ due to Theorem~\ref{LNCFormula}.

We next justify  relation \eqref{DCN}. Let $(v,v)$ satisfy \eqref{DCNTheorem-eq1} and DCNTheorem-eq2, i.e., there exist $Q\subset I(\bar x)$ and $ P\subset Q$ such that $C_Q\ne \emptyset$, $Q\vert_{I_1}\supset P\in \mathcal P$, and \eqref{DCNTheorem-eq1}  and   \eqref{DCNTheorem-eq2} hold. From the definition of $I_1(\bar x,u), I_1^+(\bar x,u)$ and \eqref{DCNTheorem-eq2}, we have
$$
P\subset I_1(\bar x,u)\subset Q\vert_{I_1}\subset I_1(\bar x,u)\cup I_1^+(\bar x,u).
$$
Since $C_Q=C_{Q\vert_{I_1}\cup Q\vert_{I_2}}\ne \emptyset,$ we have $C_{I_1(\bar x,u)\cup I_1^+(\bar x,u)\cup Q\vert_{I_2}}\ne \emptyset$, which implies that $Q\vert_{I_2}\in\mathcal I_2.$ According to \eqref{DCNTheorem-eq1}, we obtain
 \begin{align*}
        v&\in \left(\cone\{a_i\mid i\in Q\setminus P\}+{\rm span}\{a_i\mid i\in P\}\right) \cap\mathcal T_{Q\vert_{I_2}},\\
        &\subset \left(\cone\{a_i\mid i\in Q\setminus I_1(\bar x,u)\}+{\rm span}\{a_i\mid i\in I_1(\bar x,u)\}\right) \cap\left(\bigcup_{T\in\mathcal I_2}\mathcal T_T\right).
  \end{align*}
It is not difficult to see that $Q\setminus I_1(\bar x,u)\subset I_1^+(\bar x,u)\cup\bar I_2.$ Therefore, we derive that
$$
v\in \left(\cone\{a_i\mid i\in I_1^+(\bar x,u)\cup\bar I_2\}+{\rm span}\{a_i\mid i\in I_1(\bar x,u)\}\right) \cap\left(\bigcup_{T\in\mathcal I_2}\mathcal T_T\right).
$$

{\rm (ii)} We   show the equation \eqref{DCNnew} in the case that    $\{a_i\mid i\in I(\bar x)\}$ are linear independent. Let $(u,v)$ satisfy
\begin{align}\label{DCNTheorem-eq7}
\langle a_i,u\rangle=0\; \forall i\in \tilde I,\; \langle a_i,u\rangle\ge 0\;\forall j\in \Upsilon(\tilde I)\setminus\tilde I
\end{align}
and
\begin{align}\label{DCNTheorem-eq6}
v\in \big({\rm span}\,\{a_i\mid i\in I_1(\bar x,u)\} + \cone\{a_i\mid i\in I_1^+(\bar x,u)\cup \bar I_2\}\big)\cap\mathcal T_T,
\end{align}
with some $T\in \mathcal I_2.$ Under the assumed linear independence of the generating elements $\{a_i\mid i\in I(\bar x)\},$ we have $\tilde I=\Upsilon(\tilde I)=\mathcal I(\bar x,\bar x^*).$ Setting $P:=I_1(\bar x, u)$, $Q\vert_{I_2}:=T$, $Q\vert_{I_1}:=I_1(\bar x,u)\cup I_1^+(\bar x,u)$ and $Q=Q\vert_{I_1}\cup Q\vert_{I_2},$ we find from  \eqref{DCNTheorem-eq7} that $C_Q\ne\emptyset$ and $\mathcal I(\bar x,\bar x^*)\subset P$. Moreover, \eqref{DCNTheorem-eq6} implies that
\begin{align}
v\in \big(\cone\{a_i\mid i\in P\} + {\rm span}\,\{a_i\mid i\in Q\setminus P\}\big)\cap \mathcal T_{Q\vert_{I_2}},
\end{align}
which follows that $(v,-u)\in A_{P,Q}\cap \mathcal T_{Q\vert_{I_2}}\times B_{P,Q\vert_{I_1}}.$ By using Theorem~\ref{LCFormula1},  we obtain
$v\in D^*_{\mathcal C}\mathcal N(\bar x,\bar x^*)(u),$ which follows that the relation ``$\supset$'' in \eqref{DCNnew} holds. In view of  \eqref{DCN}, we derive \eqref{DCNnew} and the proof is completed.
\end{proof}

\section{Applications of formulas of the limiting corderivative with respect to a set of the normal cone mapping}
\subsection{Aubin property with respect to a set of variational inequalities}
One of the applications of the limiting coderivative with respect to a set of a multifunction is to provide equivalent characteristics for its Aubin property with respect to a set. In this subsection, by using the formulas of the limiting coderivative with respect to $\mathcal C$ of $\mathcal N$, we state the necessary and sufficient conditions for the Aubin property with respect to $\mathcal C$ of the multifunction $G:\mathbb R^n\rightrightarrows\mathbb R^n$ defined by
\begin{equation}\label{Gfunc}
    G(x):=\left\{y\mid y\in -\bar x^*+N(x,\Theta)\right\} \text{ \rm for all } x\in \R^n,
\end{equation}
where $\bar x^*\in \R^n$ is given. One observes that $y\in G(x)$ with $x\in\Theta$ if and only if $y$ is a solution to the following {\it simple variational inequalities}
$$
\langle y+\bar x^*, u-x \rangle\le 0\; \text{ \rm for all } u\in \Theta.
$$
From the definition of $G$, it is not difficult to obtain the following proposition.

\begin{proposition}\label{DCN-G}
Let $\Theta$ and $\mathcal C\subset \R^n$ be given by \eqref{Theta} and \eqref{mathcalC}, respectively, and let $G$ be defined by \eqref{Gfunc}. For any $\bar x\in \mathcal C\cap\Theta$, we have
 $$\dom D^*_{\mathcal C}G(\bar x,0)=\dom D^*_{\mathcal C}\mathcal N(\bar x,\bar x^*),$$
  $$D^*_{\mathcal C}G(\bar x,0)(u)= D^*_{\mathcal C}\mathcal N(\bar x,\bar x^*)(u)\; \text{ \rm for all } u\in \dom D^*_{\mathcal C}G(\bar x,0).$$
  Consequently, $G$ has the Aubin property with respect to $\mathcal C$ at $(\bar x,0)$ if and only if $\mathcal N$ has that property at $(\bar x,\bar x^*).$
\end{proposition}

From Theorem~\ref{thm1} or \cite[Theorem~3.6]{MWY23}, and Theorem~\ref{LCFormula1}, and Proposition~\ref{DCN-G}, we have explicit characteristics for the Aubin property with respect to $\mathcal C$ of $G$ as follows.

\begin{theorem}\label{Lip-like-con0}
Let $\Omega$ and $\mathcal C$ be respectively given by \eqref{Theta} and \eqref{mathcalC}. Let $\bar x\in \mathcal C$ and $(\bar x,\bar x^*)\in \gph \mathcal N$. The following assertions hold:

        {\rm (i)} The mapping $G$ defined by \eqref{Gfunc} satisfies the Aubin property with respect to $\mathcal C$ at $(\bar x, 0)$ if and only if the following implication holds:
        \begin{align*}
        \text{\rm for } Q\subset\bar I, P\in \mathcal P(\bar x,\bar x^*) \text{ \rm with } C_Q\ne \emptyset, P\subset Q\vert_{I_1},\\
        (v,0)\in A_{Q, P}\cap \mathcal T_{Q\vert_{I_2}}\times B_{Q\vert_{I_1},P} \implies v=0.
        \end{align*}

        {\rm (ii)} If, in addition, the generating elements $\{a_i\mid a_i\in I(\bar x)\}$ are linear independent, then $G$ satisfies the Aubin property with respect to $\mathcal C$ at $(\bar x,0)$ if and only if the following implication is fulfilled:
 \begin{align*}
        \text{\rm for } Q\subset\bar I \text{ \rm and } \mathcal I(\bar x,\bar x^*)\subset P\subset Q\vert_{I_1},\\
        (v,0)\in A_{Q, P}\cap \mathcal T_{Q\vert_{I_2}}\times B_{Q\vert_{I_1},P} \implies v=0.
        \end{align*}
\end{theorem}

From the following theorem, one sees that it is easy  to check characteristics of the Aubin property with respect to $\mathcal C$ of $G$.

\begin{theorem}\label{Lip-like-con}
Let $\Omega$ and $\mathcal C$ be respectively given by \eqref{Theta} and \eqref{mathcalC}. Let $\bar x\in \mathcal C$ and $(\bar x,\bar x^*)\in \gph \mathcal N$. The following assertions hold:

        {\rm (i)} The mapping $G$ satisfies the Aubin property with respect to $\mathcal C$ at $(\bar x,0)$ if
        \begin{equation}\label{nec-suf-lip}
        \left(\cone\{a_i\mid i\in  \bar I_2\} + {\rm span}\{a_i\mid i\in \bar I_1\}\right)\cap\left(\bigcup_{T\in\mathcal I_2}\mathcal T_T\right)=\{0\},
        \end{equation}
        i.e., whenever there exist numbers $\lambda_i\ge 0$, $\beta_j\in \R$ for all $j\in \bar I_1$, $i\in \bar I_2$ and a set $T\subset \mathcal I_2$ satisfying $$\langle a_k,\sum_{i\in \bar I_2}\lambda_ia_i+\sum_{j\in \bar I_1}\beta_ja_j\rangle\le 0\; \forall k\in T,$$ we have
        $\sum_{i\in \bar I_2}\lambda_ia_i+\sum_{j\in \bar I_1}\beta_ja_j=0.$ Here $\mathcal I_2:=\left\{T\subset \bar I_2\mid  C_{\bar I_1\cup T}\ne \emptyset\right\}.$

        {\rm (ii)} If, in addition, the generating elements $\{a_i\mid a_i\in I(\bar x)\}$ are linear independent then $G$ is the Aubin property with respect to $\mathcal C$ at $(\bar x,0)$ if and only if   relation \eqref{nec-suf-lip} holds, i.e., whenever there exist numbers $\lambda_i\ge 0$, $\beta_j\in \R$ for all $j\in \bar I_1$, $i\in \bar I_2$, and a set $T\subset \mathcal I_2$ satisfying $$\langle a_k,\sum_{i\in \bar I_2}\lambda_ia_i+\sum_{j\in \bar I_1}\beta_ja_j\rangle\le 0\; \forall k\in T,$$ we have $\lambda_i=\beta_j=0$ for all $i\in \bar I_2, j\in \bar I_1.$ Here $\mathcal I_2:=\left\{T\subset \bar I_2\mid  C_{\bar I_1\cup T}\ne \emptyset\right\}.$
\end{theorem}

\begin{proof}
The proof of the theorem is directly  from Theorem~\ref{thm1}, Theorem \ref{DCNTheorem}, Proposition~\ref{DCN-G},  and the fact that $I_1(\bar x,0)\cup I_1^+(\bar x,0)=I_1(\bar x)=\bar I_1.$
\end{proof}

We now consider examples to illustrate to above characteristics of the Aubin property with respect to $\mathcal C$ of $G$.

\begin{example}[$\mathcal N$ does not satisfy the Aubin property with respect to $\mathcal C\subset \Theta$]
Consider the following two polyhedrons
$$\Theta:=\left\{(x,y,z)\in \R^3\mid x+z\le 0\right\}$$
 and
$$\mathcal C:=\left\{(x,y,z)\in \R^3\mid x+y+z\le 0, y\ge 0\right\}.$$
Set $a_1=(1,0,1), a_2=(1,1,1)$ and $(a_3=(0,-1,0)$. Then we have $I_1=\left\{1\right\}$ and $I_2=\left\{2,3\right\}.$ It is easy to see that
$$N((x,y,z),\Theta)=
\begin{cases}
\cone\{(1,0,1)\}& \text{ \rm if } x+z = 0;\\
\{(0,0,0)\} & \text{ \rm if } x+z<0;\\
\emptyset& \text{ \rm otherwise.}
\end{cases}$$
Take $(\bar x,\bar y,\bar z)=(0,0,0)$ and $(\bar x^*,\bar y^*,\bar z^*)=(0,0,0).$ Thus
$N((\bar x,\bar y,\bar z),\Theta)=\cone\{(1,0,1)\}.$  Pick a sequence $(x_k,y_k,z_k):=(-\frac{2}{k},0,\frac{1}{k})\xrightarrow{\mathcal C}(\bar x,\bar y,\bar z)$ as $k\to\infty.$ Since $x_k+z_k<0$ and $N((x_k,y_k,z_k),\Theta)=\{(0,0,0)\}$, then there are no a neighborhood $V$ of $(\bar x^*,\bar y^*,\bar z^*)$ and a positive number $\ell$ such that
$$
N((\bar x,\bar y,\bar z),\Theta)\cap V\subset N((x_k,y_k,z_k),\Theta)+\ell\Vert(-\frac{1}{k},0,0)\Vert \mathbb B=\frac{\ell}{k}\B\to \{0\} \text{ \rm as } k\to \infty.
$$
This means that $N(\cdot,\Theta)$ does not satisfy the Aubin property with respect to $\mathcal C$ at $((0,0,0),(0,0,0)).$ Therefore, there exist, by Theorem~\ref{Lip-like-con}, $\lambda_i\ge 0, \beta_j\in \R$ for $i\in \bar I_2, j\in \bar I_1$, which satisfies $\sum_{i\in \bar I_2}\lambda_ia_i+\sum_{j\in \bar I_1}\beta_ja_j\ne 0$, and a subset $T$ of $\mathcal I_2$ such that
$$
\langle a_k,\sum_{i\in \bar I_2}\lambda_ia_i+\sum_{j\in \bar I_1}\beta_ja_j\rangle\le 0\; \forall k\in T.
$$
Indeed, take $T:=\bar I_2$ then $(-1,0,1)\in C_{\bar I_1\cup T}$, which follows that $T\in \mathcal I_2.$ Moreover, taking $\lambda_1=\lambda_2=1>0$ and $\beta=-2,$ we have
$$
\begin{cases}
\langle(0,-1,0), \lambda_1(1,1,1)+\lambda_2(0,-1,0)+\beta(1,0,1)\rangle &= 0\le 0,\\
\langle(1,1,1), \lambda_1(1,1,1)+\lambda_2(0,-1,0)+\beta(1,0,1)\rangle &= -2<0.
\end{cases}
$$
\end{example}

\begin{example}[$\mathcal N$ satisfies the Aubin property with respect to $\mathcal C\subset \Theta$]
Consider the following two polyhedrons
$$\Theta:=\left\{(x,y,z)\in \R^3\mid x+z\le 0\right\}$$
 and
$$\mathcal C:=\left\{(x,y,z)\in \R^3\mid x+y+z\le 0, y\ge 0, x+z\ge 0\right\}.$$ We denote $a_1=(1,0,1), a_2=(1,1,1), a_3=(0,-1,0)$ and $a_4=(-1,0,-1).$
Then we have $I_1=\left\{1\right\}$, $I_2=\left\{2,3,4\right\}$, and
$$N((x,y,z),\Theta)=
\begin{cases}
\cone\{(1,0,1)\}& \text{ \rm if } x+z = 0;\\
\{(0,0,0)\} & \text{ \rm if } x+z<0;\\
\emptyset& \text{ \rm otherwise.}
\end{cases}$$
Take $(\bar x,\bar y,\bar z)=(0,0,0)$ and $(\bar x^*,\bar y^*,\bar z^*)=(0,0,0).$ By the directly computation, we have $\mathcal I_2=\{\bar I_2\}.$ For any $\lambda_i\ge 0, \beta_j\in \R$ for $i\in \bar I_2$ and $i\in \bar I_1,$ satisfying the following inequalities
$$
\begin{cases}
\langle a_k,\sum\limits_{i\in \bar I_2}\lambda_ia_i+\beta_1a_1 \rangle\le 0\; \forall k\in \bar I_2, \\
\lambda_i\ge 0\; \forall i=1,2,3,
\end{cases}
$$
we have $\lambda_1-\lambda_2=\lambda_1-\lambda_3+\beta_1=0$,  so $\sum_{i\in\bar I_2}\lambda_ia_i+\beta_1a_1=0.$ Using Theorem~\ref{Lip-like-con}, we obtain that $N(\cdot, \Theta)$ satisfies the Aubin property with respect to $\mathcal C$ at $((\bar x,\bar y,\bar z),(\bar x^*,\bar y^*,\bar z^*)).$ We check this assertion by the definition. For any $(x_1,y_1,z_1), (x_2,y_2,z_2)\in \mathcal C\cap \B((0,0,0),1)$, we have $x_1+z_1=x_2+z_2=0$, so
$$N((x_1,y_1,z_1),\Theta)=N((x_2,y_2,z_2),\Theta)=\cone\{(1,0,1)\},$$ which  implies that
$$N((x_1,y_1,z_1),\Theta)\cap \mathbb B((0,0,0),1)\subset N((x_2,y_2,z_2),\Theta)+\Vert (x_1,y_1,z_1)-(x_2,y_2,z_2)\Vert\mathbb B.$$ Hence, $N(\cdot, \Theta)$ really satisfies the Aubin property with respect to $\mathcal C$ at $((\bar x,\bar y,\bar z),(\bar x^*,\bar y^*,\bar z^*)).$
\end{example}

In the case that $\mathcal C\nsubseteq \Theta,$ Theorem~\ref{Lip-like-con} still work very well. We consider the following example.

\begin{example}[The set $\mathcal C\nsubseteq \Theta$]
Consider the following two polyhedrons
$$\Theta:=\left\{(x,y,z)\in \R^3\mid x+z\le 0\right\}$$
 and
$$\mathcal C:=\left\{(x,y,z)\in \R^3\mid x+y+z\le 0, z\ge 0\right\}.$$
Setting $a_1=(1,0,1), a_2=(1,1,1)$ and $a_3=(0,0,-1)$, we have $I_1=\left\{1\right\}$, $I_2=\left\{2,3\right\}$ and
$$N((x,y,z),\Theta)=
\begin{cases}
\cone\{(1,0,1)\}& \text{ \rm if } x+z = 0;\\
\{(0,0,0)\} & \text{ \rm if } x+z<0;\\
\emptyset& \text{ \rm otherwise.}
\end{cases}$$
Take $(\bar x,\bar y,\bar z)=(0,0,0)$ and $(\bar x^*,\bar y^*,\bar z^*)=(0,0,0).$ We have $$N((\bar x,\bar y,\bar z),\Theta)=\cone\{(1,0,1)\}.$$ Pick a sequence $(x_k,y_k,z_k):=(-\frac{1}{k},0,\frac{2}{k})\xrightarrow{\mathcal C}(\bar x,\bar y,\bar z)$ as $k\to\infty.$ Since $x_k+z_k>0,$ $N((x_k,y_k,z_k),\Theta)=\emptyset$, then there are no a neighborhood $V$ of $(\bar x^*,\bar y^*,\bar z^*)$ and a positive number $\ell$ such that
$$N((\bar x,\bar y,\bar z),\Theta)\cap V\subset N((x_k,y_k,z_k),\Theta)+\ell\Vert(-\frac{1}{k},0,0)\Vert \mathbb B=\emptyset.$$
This follows that $N(\cdot,\Theta)$ does not satisfy the Aubin property with respect to $\mathcal C$ at $((0,0,0),(0,0,0)).$ Therefore, there exist $\lambda_i\ge 0, \beta_j\in \R$ for $i\in \bar I_2, j\in \bar I_1$, which are not all equal to zero, and a subset $T$ of $\mathcal I_2$ such that
$$\langle a_k,\sum_{i\in \bar I_2}\lambda_ia_i+\sum_{j\in \bar I_1}\beta_ja_j\rangle\le 0,\; \forall k\in T.$$
We show the assertion above by the directly computation. It is easy to see that if $T:=\{(0,0,-1)\}$, then $(-1,-1,1)\in C_{\bar I_1\cup T}$, which follows that $T\in \mathcal I_2.$ Moreover, taking $\lambda_1=\lambda_2=\beta=1> 0,$ we have
$$\langle(0,0,-1), \lambda_1(1,1,1)+\lambda_2(0,0,-1)+\beta(1,0,1)\rangle = -1<0.$$
\end{example}

\subsection{Optimality conditions for robust solutions to bilevel programming with linear lower level problem with uncertain }
In this section, we provide necessary conditions for optimal solutions of the following bilevel programming problems:
\begin{align}\label{BLP}
    \min\; &f(x)\tag{SBLP}\\
    \text{ \rm such that } & x\in \mathcal S\cap \mathcal C\nonumber\\
    \text{ \rm with }& \mathcal S:={\rm arcmin}\, \{c^{\top}x\mid x\in \Theta\},\nonumber
\end{align}
where $c\in \R^n,$ $f:\Theta\to \R$, and $\Theta, \mathcal C\subset \R^n$ are given by \eqref{Theta} and \eqref{mathcalC}, respectively.  Problem \eqref{BLP} can be written in the form of  problem \eqref{mpecs} as follows:
\begin{align}\label{BLP1}
    \min\; &f(x)\tag{MPEC}\\
    \text{ \rm such that } & 0\in G(x) \text{ \rm and } x\in \mathcal C_1\cap\mathcal C_2,\nonumber
\end{align} where $\mathcal C_1:=\mathcal C_2:=\mathcal C$ and $G(x):=\left\{y\mid y\in c+N(x,\Theta)\right\}$ for all $x\in \R^n.$

We consider the mixed problem of the \ref{BLP1} problem as follows.
\begin{align}\label{BLP2}
    \min\; &\tilde f(x,y)\tag{MMPEC}\\
    \text{ \rm such that } & y\in G(x) \text{ \rm and } x\in \mathcal C_1\cap\mathcal C_2,\nonumber
\end{align} where $\tilde f:\R^n\times\R^m\to \R$ is defined by $\tilde f(x,y):=f(x)$ for all $y\in \R^m.$ It is not difficult to see that $\bar x$ is a local solution to the \ref{BLP1} problem whenever $(\bar x,0)$ is a local solution to the \ref{BLP2} problem. However, the opposite assertion may not hold in general. The local solution $\bar x$ to the \ref{BLP1} problem is called to be a {\it robust local solution} to the \ref{BLP1} problem (and the \ref{BLP} problem also) if $(\bar x,0)$ is a local solution to the \ref{BLP2} problem.

Take $\bar x^*:=-c$ and use the notions $\bar I,\bar I_i$ for $i=1,2$ and so all as Section~4.  By using Theorem~\ref{nec-con}, we now provide optimality conditions for a robust local solution to the problem \eqref{BLP} under some qualification conditions as follows.

\begin{theorem}\label{optim-1}
    Consider  problem \eqref{BLP}. Let $\bar x$ be a robust local solution to problem \eqref{BLP}, and let the following assumptions hold:

    {\rm (i)} $[-\partial^{\infty}_{\mathcal C}f(\bar x)]\cap \left(\cone\{a_i\mid i\in  \bar I_2\} + {\rm span}\{a_i\mid i\in \bar I_1\}\right)\cap\left(\bigcup_{T\in\mathcal I_2}\mathcal T_T\right)= \{0\},$ where
    $$\mathcal I_2:=\left\{T\subset \bar I_2\mid  C_{\bar I_1\cup T}\ne \emptyset\right\}.$$

    {\rm (ii)} $\{\Omega_1,\Omega_2\}$ are normal-densed in $\{\mathcal C_1,\mathcal C_2\}$ at $(\bar x, 0 , f(\bar x)),$ where $$\Omega_1:=\left\{(x,y,z)\mid (x,z)\in\epi f, y\in \R^n\right\}$$ and $$\Omega_2:=\gph G\times\R.$$

    Then
    $$0\in \partial_{\mathcal C} f(\bar x) + \left(\cone\{a_i\mid i\in  \bar I_2\} + {\rm span}\{a_i\mid i\in \bar I_1\}\right)\cap \left(\bigcup_{T\in\mathcal I_2}\mathcal T_T\right).$$
\end{theorem}

It known from Theorem~\ref{thm2} that   condition (i) is fulfilled if $f$ is local Lipschitz-like with respect to $\mathcal C$ at $\bar x$. Moreover, (i) is also satisfied if   relation \eqref{nec-suf-lip} holds. In this case, we have the simple necessary optimality condition as the following corollary.

\begin{corollary}
    Consider  problem \eqref{BLP}. Let $\bar x$ be a robust local solution to problem \eqref{BLP}, and let   relation \eqref{nec-suf-lip} hold. Assume that {\rm (ii)} in Theorem~\ref{optim-1} is fulfilled. Then $0\in \partial_{\mathcal C} f(\bar x).$
\end{corollary}

We next establish necessary conditions for a local solution to the \ref{BLP} problem as follows.

\begin{theorem}\label{optim-3}
    Consider  problem \eqref{BLP}. Let $\bar x$ be a local solution to problem \eqref{BLP}, and let the following assumptions hold:

    {\rm (i)} $[-\partial^{\infty}_{\mathcal C}f(\bar x)]\cap \left(\cone\{a_i\mid i\in  \bar I_2\} + {\rm span}\{a_i\mid i\in \bar I_1\}\right)\cap\left(\bigcup_{T\in\mathcal I_2}\mathcal T_T\right)= \{0\},$ where
    $$\mathcal I_2:=\left\{T\subset \bar I_2\mid  C_{\bar I_1\cup T}\ne \emptyset\right\}.$$

    {\rm (ii)} $\{\Omega_1,\Omega_2\}$ are normal-densed in $\{\mathcal C_1,\mathcal C_2\}$ at $(\bar x, 0 , f(\bar x)),$ where $$\Omega_1:=\left\{(x,y,z)\mid (x,z)\in\epi f, y\in \R^n\right\}$$ and $$\Omega_2:=\gph G\times\R.$$

{\rm (iii)} There exists $\delta>0$ such that for any $x\in \mathcal C\cap\B(\bar x,\delta),$ the following implication holds:
\begin{equation}
    \label{optim-3-eq1}
   d(0, G(x))>0\implies d(f(\bar x),f(x)+\R_+)>0.
\end{equation}

    Then
    $$0\in \partial_{\mathcal C} f(\bar x) + \left(\cone\{a_i\mid i\in  \bar I_2\} + {\rm span}\{a_i\mid i\in \bar I_1\}\right)\cap \left(\bigcup_{T\in\mathcal I_2}\mathcal T_T\right).$$
\end{theorem}
\begin{proof}
    Let $\bar x$ be a local solution to the \ref{BLP} problem. There exists $\bar\delta\le \delta$ such that
    $$f(\bar x)\le f(x)\; \text{ \rm for all } x\in \mathcal C\cap\B(\bar x,\delta) \text{ \rm with } 0\in G(x).$$ For any $(x,y)\in \mathcal C\times\R^{m}\cap\B((\bar x,0),\bar\delta)$ with $y\in G(x),$ we consider the following two cases.

    Case 1. If $0\in G(x)$ then $x\in \mathcal C\cap \B(\bar x,\delta)$ and $0\in G(x).$ Thus, we have
    $$\tilde f(\bar x,0)=f(\bar x)\le f(x)=\tilde f(x,y).$$

    Case 2. If $0\notin G(x)$ then $d(0,G(x))>0.$ By \eqref{optim-3-eq1}, we have $d(f(\bar x),f(x)+\R)>0.$ It implies that $f(\bar x)\notin \{f(x)\}+\R_+$ which gives us that $$\tilde f(\bar x,0)=f(\bar x)\le f(x)=\tilde f(x,y).$$

    Combining both two cases, we obtain $$\tilde f(\bar x,0)=f(\bar x)\le f(x)=\tilde f(x,y)\; \text{ \rm for all } (x,y)\in \B((\bar x,0),\delta) \text{ \rm with } y\in G(x).$$ So $(\bar x,0)$ is a local solution to the \ref{BLP2} problem and then $\bar x$ is a robust local solution to the \ref{SBLP} problem. By using Theorem~\ref{optim-1}, we obtain that
    $$0\in \partial_{\mathcal C} f(\bar x) + \left(\cone\{a_i\mid i\in  \bar I_2\} + {\rm span}\{a_i\mid i\in \bar I_1\}\right)\cap \left(\bigcup_{T\in\mathcal I_2}\mathcal T_T\right).$$
\end{proof}

\begin{remark}
{\rm   The \ref{BLP1} problem is similar to the problem (6.1) in \cite{MO2007} but the loss of continuous differentiability of the objective function and so \cite[Theorem 6.1]{MO2007} may not be applied to this problem.  Problem \eqref{BLP1} can be written in the form of the problems in \cite{G13, TC18}, so we can use directionally necessary optimality conditions as in those papers to study  problem~\eqref{BLP1}. However, in that case, necessary conditions are not explicit and may need to add other qualification conditions, for example, $G$ is metrically subregular in critical directions at $\bar x$ and/or $f$ is local Lipschitz at $\bar x.$}
\end{remark}

We close this section by an illustrated example for Theorem~\ref{optim-3}.

\begin{example}
    Consider the following problem:
    \begin{align}
        \min\; & \sqrt{\vert x\vert}\label{prob4}\\
        \text{ \rm such that }& x\in \mathcal S, x\le 0,\nonumber\\
        \textbf{\rm where }& \mathcal S:={\rm argmin}\{x\mid x\ge 0\}.\nonumber
    \end{align}
    Setting $\Theta:=\left\{x\in \R\mid x\ge 0\right\}$ and $\mathcal C:=\left\{x\in\R\mid x\le 0\right\}$, we have $a_1:=-1, a_2:=\{1\}$ and $b_i=0$ for all $i\in I_1\cup I_2$ with $I_1=\{1\}, I_2=\{2\}.$ Moreover, it is easy to see that $\bar x=0$ is a unique local solution to the problem. By directly computations, we have $\bar I_1=\{1\},\bar I_2=\{2\}$ and $\mathcal I_2=\{\bar I_2\}$,
$$\partial^{\infty}_{\mathcal C}f(0)=\partial_{\mathcal C}f(0)= \mathcal T_{\bar I_2}=\R_- \; \text{ \rm and }\;  \cone\{a_i\mid i\in \bar I_2\}+{\rm span}\,\{a_i\mid i\in \bar I_1\}=\R,$$
where $\bar I_1=\{-1\}$ and $\bar I_2=\{1\}.$ Thus $(q_1)$ holds. We next consider the following sets.
$$\Omega_1:=\{(x,y,z)\mid (x,z)\in \epi f, y\in \R\} \text{ \rm and } \Omega_2:=\left\{(x,y,z)\mid (x,y)\in \gph G\right\},$$
where $G(x):=\{y\mid y\in \{1\} + N(x,\Theta) \}=\begin{cases} (-\infty ,1]& \text{ \rm if } x =0,\\ \{1\} &\text{ \rm if } x>0,\\ \emptyset &\text{ otherwise.} \end{cases}$

For any $(x,y)\xrightarrow{\tgph G\cap(\R_-\times\R)}(0,0)$ and $(x,z)\xrightarrow{(\tepi f)\cap(\R_-\times\R)}(0,0)$, we have
$$\hat N((x,y),\gph G\cap(\R_-\times\R))=\hat N_{\R_-\times\R}((x,y),\gph G)=\begin{cases}
\{0\}\times \R_- &\text{ \rm if } x=0,\\
\emptyset &\text{ \rm otherwise,}\end{cases}
$$
$$\hat N((x,y),\epi f\cap(\R_-\times\R))=\begin{cases}
\R_-(\frac{1}{2\sqrt{-x}},1)& \text{ \rm if } x<0,z=\sqrt{-x},\\
\R\times\R_-&\text{ \rm if } (x,z)=(0,0),\\
\R_+\times\{0\}&\text{ \rm if } x=0,z>0,\\
\{0\}^2&\text{ \rm if } x<0, z>\sqrt{-x},
\end{cases}$$
$$\hat N_{\R_-\times\R}((x,z),\epi f)=\begin{cases}
\R_-(\frac{1}{2\sqrt{-x}},1)& \text{ \rm if } x<0,z=\sqrt{-x},\\
\R_-\times\R_-&\text{ \rm if } (x,z)=(0,0),\\
\{0\}^2&\text{ \rm if } x=0,z>0 \text{ \rm or } x<0, z>\sqrt{-x}.
\end{cases}$$ Thus it is not difficult to see that $\{\Omega_1,\Omega_2\}$ are normal-densed in $\{\R_-\times\R^2,\R_-\times\R^2\}$, which means that $(q_2)$ holds.Moreover, \eqref{optim-3-eq1} holds for any $\delta>0.$ So we have
$$0\in \partial_{\mathcal C}f(\bar x)+(\cone\{a_i\mid i\in \bar I_2\}+{\rm span}\,\{a_i\mid i\in \bar I_1\})\cap\mathcal T_{\bar I_2}.$$
By directly calculations, we have
$$\partial_{\mathcal C}f(\bar x)+(\cone\{a_i\mid i\in \bar I_2\}+{\rm span}\,\{a_i\mid i\in \bar I_1\})\cap\mathcal T_{\bar I_2}=\R_-.$$
  It is necessary to note that the constraints of the problem \eqref{prob4} can be written by $x\in \mathcal S$ with $\mathcal S:={\rm argmin}\{x\mid x=0\}$. In this case, the qualification condition in \cite[Theorem~1]{DDD2010} is not fulfilled. Indeed, we have $N(0,\Theta)=\R$ and $N((0,0),\gph \mathcal N)=\R\times\{0\}$ and thus $(w,v)=(1,0)\ne (0,0)$ satisfies
$$(w,v)\in N((0,0),\gph \mathcal N)\; \text{ \rm and } 0\in w-\nabla^2h(0)^{\top}w +N(0,\Theta).$$
\end{example}

\end{document}